\documentclass[a4paper,11pt,UKenglish]{article}
\usepackage[top=3.2cm, bottom=3cm, left=3.45cm, right=3.45cm]{geometry}
\usepackage[T1]{fontenc}
\usepackage[utf8]{inputenc}
\usepackage{latexsym}
\usepackage{amsmath,amsthm,amssymb,amsfonts}
\usepackage{mathtools}
\usepackage[UKenglish]{babel}
\usepackage{tikz}
\usetikzlibrary{patterns, matrix}
\tikzstyle{NE-lines}=[pattern=north east lines, pattern color=black!45]
\usepackage{booktabs}

\newcommand{\etal}{et~al.}
\newcommand{\Ascseq}{A}                   
\newcommand{\Modasc}{\hat{\Ascseq}}       
\newcommand{\twoplustwo}{\mathbf{2\hspace{-0.2em}+\hspace{-0.2em}2}}
\newcommand{\Prim}{\Modasc^{\mathrm{pr}}} 
\newcommand{\F}{F}                        
\newcommand{\Cay}{\mathrm{Cay}}           
\newcommand{\Sym}{\mathrm{Sym}}                      
\newcommand{\RGF}{\mathrm{RGF}}           

\newcommand{\asc}{\mathrm{asc}}           
\newcommand{\des}{\mathrm{des}}           
\newcommand{\asctops}{\mathrm{top}}       
\newcommand{\nub}{\mathrm{nub}}           
\newcommand{\patta}{\mathfrak{a}}
\newcommand{\pattb}{\mathfrak{b}}
\newcommand{\Par}{\mathrm{Par}}

\mathchardef\mhyphen="2D
\newcommand{\dashpatt}{32\mhyphen 1}
\newcommand{\std}{\mathfrak{st}}
\newcommand{\D}{\mathtt{d}}
\newcommand{\U}{\mathtt{u}}
\newcommand{\dudu}{\D\U\D\U}
\newcommand{\ltrmax}{\mathrm{lrmax}}
\newcommand{\ltrmin}{\mathrm{lrmin}}
\newcommand{\wltrmax}{\mathrm{wlrmax}}
\newcommand{\wltrmin}{\mathrm{wlrmin}}
\newcommand{\rtlmax}{\mathrm{rlmax}}
\newcommand{\rtlmin}{\mathrm{rlmin}}
\newcommand{\wrtlmax}{\mathrm{wrlmax}}
\newcommand{\wrtlmin}{\mathrm{wrlmin}}
\newcommand{\ogf}{{\textsc{ogf}}}

\DeclareMathOperator{\img}{Im}

\newtheorem{theorem}{Theorem}[section]
\newtheorem{theorem*}{Theorem}[section]
\newtheorem{proposition}[theorem]{Proposition}
\newtheorem{lemma}[theorem]{Lemma}
\newtheorem{corollary}[theorem]{Corollary}

\newtheorem*{openproblem*}{Open Problem}

\theoremstyle{definition}

\newtheorem{remark}[theorem]{Remark}
\newtheorem*{remark*}{Remark}

\newtheorem*{example*}{Example}

\title{Pattern-avoiding modified ascent sequences}
\author{
Giulio Cerbai\footnote{The author is a member of the INdAM
research group GNCS.}\\
\emph{Science Institute, University of Iceland}\\
\emph{107 Reykjavik, Iceland}\\
\texttt{giuliocerbai14@gmail.com}
}
\date{}

\begin{document}
\maketitle

\begin{abstract}
We initiate an in-depth study of pattern avoidance on modified
ascent sequences. Our main technique consists in using Stanley's
standardization to obtain a transport theorem between primitive
modified ascent sequences and permutations avoiding
a bivincular pattern of length three. We enumerate some patterns
via bijections with other combinatorial structures such as Fishburn
permutations, lattice paths and set partitions.
We settle the last remaining case of a conjecture by
Duncan and Steingr\'imsson by proving that modified
ascent sequences avoiding~$2321$ are counted by the Bell numbers.
\end{abstract}

\section{Introduction}

Modified ascent sequences have recently assumed a central role
in the study of Fishburn structures. Originally~\cite{BMCDK}, they
were defined as the bijective image of (plain) ascent sequences under
a certain hat map, with the primary role of making their relation with
$(\twoplustwo)$-free posets more transparent. More recently, Claesson
and the current author~\cite{CC} introduced the Burge transpose to
develop a theory of transport of patterns between modified ascent
sequences and Fishburn permutations, defined as those avoiding a certain
bivincular pattern of length three. They also characterized modified
ascent sequences as Cayley permutations where each entry is a
leftmost copy if and only if it sits at an ascent top (see also
Proposition~\ref{prop_modasc_prop}).
This alternative description---not relying on the hat map---opened
the door for a study of modified ascent sequences as independent
objects, under both a geometrical and enumerative perspective.
Ultimately, it led to the introduction by the same authors of
Fishburn trees~\cite{CC2}. This class of binary, labeled trees
originates from the max-decomposition of modified ascent sequences.
Conversely, modified ascent sequences are obtained by reading the
labels of Fishburn trees with the in-order traversal.
The relation between Fishburn trees and other Fishburn structures,
namely Fishburn matrices and $(\twoplustwo)$-free posets, is
extremely transparent. For instance, Fishburn matrices arise by
decomposing a Fishburn tree with respect to its maximal right-paths.
The reader who is interested in the state of the art on Fishburn
structures is referred to the same paper~\cite{CC2}.

Motivated by all the above reasons, we conduct a more
systematic study of pattern avoidance on modified ascent sequences,
using a variety of combinatorial tools and methods.
Our investigation is parallel to the one by Duncan and Steingr\'imsson
on plain ascent sequences~\cite{DS}.
Given a pattern~$y$, our goal is to ``solve'' it by counting the number
of modified ascent sequences of given length that avoid~$y$.
Here, to count means to obtain an explicit formula, when possible,
a generating function, or a bijection with another combinatorial
structure whose enumeration is known.
An overview of our results can be found in Table~\ref{table_patts}.
Our main technique relies on what could be merely regarded
as a ``trick''---one that is unexpectedly effective in practical
terms. Namely, we study primitive ascent sequences first, defined
in Section~\ref{section_primitive} as those with no pairs of
consecutive equal entries. We show that Stanley's
standardization~\cite{St} maps bijectively primitive modified
ascent sequences to the set~$\Omega$ of permutations that start with~$1$
and avoid the bivincular pattern~$\omega$, defined in
Section~\ref{section_std}.
As a result, we obtain in Theorem~\ref{thm_transport}
a mechanism to transport patterns between primitive modified ascent
sequences and~$\Omega$. The main advantage of this approach is that
it often allows us to work with permutations, a task that is much easier
due to the arsenal of tools at our disposal.
Finally, as showed in Proposition~\ref{prop_prim_to_full}, by applying
a simple binomial transform to the counting sequence of primitive
words we immediately obtain the enumeration of the general case.

Let us end this preamble with a more detailed presentation of
our paper.

\begin{table}
\centering
\def\arraystretch{1}
\begin{tabular}{llll}
\toprule
$y$ & $|\Modasc_n(y)|$ & $|\Prim_n(y)|$ & Reference\\
\midrule
11 & $1,1,1,\dots$ & $1,1,1,\dots$ & Section~\ref{sec_patt_len2}\\
12 & $1,1,1,\dots$ & $1,0,0,\dots$ & Section~\ref{sec_patt_len2}\\
21,121 & $2^{n-1}$ & $1,1,1,\dots$ & Section~\ref{sec_patt_len2}\\
\midrule
112 & $2^{n-1}$ & Fibonacci & Section~\ref{sec_112}\\
122 & A026898 & A229046 & Section~\ref{section_122}\\
123 & $2^{n-1}$ & $1,1,1,\dots$ & Section~\ref{section_123}\\
132 & Odd Fibonacci & Fibonacci & Section~\ref{section_132}\\
212,1212 & Bell & Bell (shifted) & \cite{Ce2}\\
213,1213 & Catalan & Motzkin &
Sections~\ref{section_213_231_321},\;\ref{section_1213_1312}\\
221 & New & Bell (shifted) & Section~\ref{section_221_2321}\\
231 & Catalan & Motzkin & Section~\ref{section_213_231_321}\\
312,1312 & New & A102407 &
Sections~\ref{section_312},\;\ref{section_1213_1312}\\
321 & A007317 & Catalan & Section~\ref{section_213_231_321}\\
\midrule
1123  & Catalan & A082582? & \cite{CC}\\
1232 & A047970 & A229046 & Section~\ref{section_1232}\\
1234 & Catalan & Motzkin & Section~\ref{section_123}\\
2132 & Bell & Bell (shifted) & \cite{Ce2}\\
2213 & Bell & ? & \cite{Ce2}\\
2231 & Bell & ? & \cite{Ce2}\\
2321 & Bell & Bell (shifted) & Section~\ref{section_221_2321}\\
\bottomrule
\end{tabular}
\caption{Enumeration of modified ascent sequences avoiding a
single pattern~$y$. The counting sequences start from $n=1$.
Patterns in the same row determine the same set of sequences,
while a question mark denotes numerical data that we were not able
to confirm.}\label{table_patts}
\end{table}

In Section~\ref{sec_prel}, we give a short introduction to
permutation patterns and define (primitive) modified ascent
sequences. Then, we prove in Proposition~\ref{prop_prim_to_full}
that if~$y$ is a primitive pattern, then modified ascent sequences
avoiding~$y$ are counted by a binomial transform of their primitive
counterpart.

In Section~\ref{section_std}, we recall the definition of Stanley's
standardization and prove some related properties. The main result
of this section, Theorem~\ref{thm_transport}, is the theorem of
transport between~$\Omega$ and the set of primitive modified ascent sequences mentioned previously.

In Section~\ref{section_easy}, we enumerate modified ascent sequences
avoiding any pattern of length two, as well as a couple of simple
patterns of length three. We give a bijection between modified
ascent sequences avoiding~$122$ and set partitions whose minima
of blocks form an interval, computing some related generating
functions in the process.

In Section~\ref{section_prim_patts}, we solve several primitive patterns
with the machinery of Proposition~\ref{prop_prim_to_full} and
Theorem~\ref{thm_transport}. The hardest one is~$312$, which
we settle by showing a bijection with Dyck paths avoiding the
consecutive subpath~$\dudu$. Our construction is based on a
geometric decomposition of Dyck paths that leads to a generating
function first discovered by Sapounakis, Tasoulas and
Tsikouras~\cite{STT}.

In Section~\ref{section_221_2321}, we slightly tweak
Proposition~\ref{prop_prim_to_full} to solve the pattern~$221$, which
is not primitive. En passant, we prove in Proposition~\ref{conj_DS}
that modified ascent sequences avoiding~$2321$ are enumerated by the
Bell numbers, settling the last remaining case of a conjecture first
proposed by Duncan and Steingr\'imsson~\cite{DS} and solved only
partially by the current author~\cite{Ce2}.

In Section~\ref{section_final}, we provide some data for the
unsolved patterns and leave some suggestions for future work.

\section{Preliminaries}\label{sec_prel}

Given a natural number~$n\ge 0$, let $[n]=\lbrace 1,2,\dots, n\rbrace$.
An \emph{endofunction} of size $n$ is a map $x:[n]\to[n]$.
We shall identify~$x$ with the word $x=x_1\dots x_n$,
where $x_i=x(i)$ for each $i\in[n]$.
When $n=0$, we identify the empty endofunction with the empty word.
A \emph{Cayley permutation}~\cite{Ca,MF} is an endofunction
$x:[n]\to[n]$ whose image is $\img(x)=[k]$, for some $k\le n$.
In other words, an endofunction~$x$ is a Cayley permutation if it
contains at least a copy of every integer between $1$ and its
maximum value.
For the rest of this paper, if $A$ is a set whose elements are
equipped with a notion of size, we will denote with $A_n$ the set
of elements in $A$ of size $n$. Conversely, given a definition
of $A_n$ (of elements of size $n$) we assume $A=\cup_{n\geq 0}A_n$.
As an example, we define the set of Cayley permutations
of size $n$ as $\Cay_n$ and let $\Cay=\cup_{n\geq 0}\Cay_n$.
A Cayley permutation $x=x_1\cdots x_n$ with $\max(x)=k$ encodes the
ordered set partition $B_1\dots B_k$, where $i\in B_{x_i}$.
The map defined this way is bijective, and for this reason
Cayley permutations are counted by the Fubini numbers (listed as 
sequence A000670 in the OEIS~\cite{Sl}).

A \emph{left-to-right minimum} (briefly, $\ltrmin$) of
$x=x_1\cdots x_n$ is a pair $(i,x_i)$ such that
$x_i<\min(x_1\cdots x_{i-1})$. If we replace the strict inequality
with a weak one, i.e. if $x_i\le\min(x_1\cdots x_{i-1})$, then $(i,x_i)$
is said to be a \emph{weak left-to-right minimum} (briefly, $\wltrmin$).
We denote the set of $\ltrmin$ and $\wltrmin$ of $x$ respectively
by $\ltrmin(x)$ and $\wltrmin(x)$. Left-to-right maxima, right-to-left
minima and maxima, as well as their weak counterparts, are
defined analogously. When there is no ambiguity, we omit the
index~$i$ from the pair $(i,x_i)$. For instance, we sometimes write
$\wrtlmax(x)=\{x_i: x_i\ge x_j\text{ for each $j>i$}\}$.

An comprehensive introduction to permutation patterns can be found in
the book by Kitaev~\cite{Ki}. Bevan's note~\cite{Be} contains a
brief presentation of the most used notions and definitions
in the permutation patterns field. Below, we quickly recall
those that are necessary in this paper.

Let $x\in\Cay_n$ and $y\in\Cay_k$ be two Cayley permutations,
with $k\le n$. We say that~$x$ \emph{contains}~$y$ if there is a
subsequence $x_{i_1}x_{i_2}\cdots x_{i_k}$, with $i_1<i_2<\cdots<i_k$,
that is \emph{order isomorphic} to $y$. Here, order isomorphic means
that $x_{i_s}<x_{i_t}$ if and only if $y_s<y_t$, and $x_{i_s}=x_{i_t}$
if and only if $y_s=y_t$. In this case, we write $x\ge y$ and
$x_{i_1}x_{i_2}\cdots x_{i_k}\simeq y$; further, the subsequence
$x_{i_1}x_{i_2}\cdots x_{i_k}$ is an \emph{occurrence} of the
\emph{pattern} $y$ in $x$. If no subsequence of $x$ is order isomorphic
to $y$, we say that $x$ \emph{avoids}~$y$. Given a pattern~$y$, we
let $\Cay(y)$ be the set of Cayley permutations that avoid $y$.
More in general, when $B$ is a set of patterns, $\Cay(B)$ shall denote
the set of Cayley permutations avoiding every pattern in $B$.
We use analogous notations for subsets of $\Cay$, as well as for
other types of pattern. For instance, $\Modasc(112)$ denotes the
set of modified ascent sequences (defined in Section~\ref{section_modasc})
avoiding the pattern $112$. The set of \emph{permutations} (i.e. bijective
endofunctions) is defined via pattern avoidance as $\Sym=\Cay(11)$.

Classical patterns are generalized by mesh patterns and Cayley-mesh
patterns. A \emph{mesh pattern}~\cite{Cl} is a pair $(y,R)$, where
$y\in\Sym_k$ is a permutation (classical pattern) and
$R\subseteq\left[0,k\right]\times\left[0,k\right]$ is a set of pairs
of integers. The pairs in $R$ identify the lower left corners of unit
squares in the plot of $x$ which specify forbidden regions. An occurrence
of the mesh pattern $(y,R)$ in the permutation $x$ is an occurrence of
the classical pattern $y$ such that no other points of $x$ occur in the
forbidden regions specified by~$R$. By allowing additional regions for
repeated entries, we arrive at \emph{Cayley-mesh patterns}~\cite{Ce};
that is, mesh patterns on Cayley permutations. To ease notation, we
often define a (Cayley-)mesh pattern~$(y,R)$ by simply plotting the underlying
classical pattern~$y$, with the forbidden regions determined by~$R$
shaded. An interesting example is the following. Claesson and the current
author~\cite{CC} characterized the set~$\Modasc$ of \emph{modified ascent
sequences} as~$\Modasc=\Cay(\patta,\pattb)$, where $\patta$
and~$\pattb$ are defined by Figure~\ref{figure_mesh_patts}.

\begin{figure}
$$
\patta \,=\,
\begin{tikzpicture}[scale=0.50, baseline=19pt]
\fill[NE-lines] (2.15,0) rectangle (2.85,3);
\draw [semithick] (0,0.85) -- (4,0.85);
\draw [semithick] (0,1.15) -- (4,1.15);
\draw [semithick] (0,1.85) -- (4,1.85);
\draw [semithick] (0,2.15) -- (4,2.15);
\draw [semithick] (0.85,0) -- (0.85,3);
\draw [semithick] (1.15,0) -- (1.15,3);
\draw [semithick] (1.85,0) -- (1.85,3);
\draw [semithick] (2.15,0) -- (2.15,3);
\draw [semithick] (2.85,0) -- (2.85,3);
\draw [semithick] (3.15,0) -- (3.15,3);
\filldraw (1,2) circle (5pt);
\filldraw (2,1) circle (5pt);
\filldraw (3,2) circle (5pt);
\end{tikzpicture}
\qquad
\pattb \,=\,
\begin{tikzpicture}[scale=0.50, baseline=19pt]
\fill[NE-lines] (1.15,0) rectangle (1.85,3);
\fill[NE-lines] (0,0.85) rectangle (0.85,1.15);
\draw [semithick] (0,0.85) -- (3,0.85);
\draw [semithick] (0,1.15) -- (3,1.15);
\draw [semithick] (0,1.85) -- (3,1.85);
\draw [semithick] (0,2.15) -- (3,2.15);
\draw [semithick] (0.85,0) -- (0.85,3);
\draw [semithick] (1.15,0) -- (1.15,3);
\draw [semithick] (1.85,0) -- (1.85,3);
\draw [semithick] (2.15,0) -- (2.15,3);
\filldraw (1,2) circle (5pt);
\filldraw (2,1) circle (5pt);
\end{tikzpicture}
$$
\caption{Cayley-mesh patterns such that $\Modasc=\Cay(\patta,\pattb)$.
}\label{figure_mesh_patts}
\end{figure}

\subsection{Modified ascent sequences}\label{section_modasc}

Recall from the end of the previous section that the set~$\Modasc$ of
modified ascent sequences is $\Modasc=\Cay(\patta,\pattb)$, where~$\patta$
and~$\pattb$ are depicted in Figure~\ref{figure_mesh_patts}.
Let us point out that this definition departs slightly from the original
one~\cite{BMCDK}: our sequences are $1$-based instead of being
$0$-based. Below we recall two useful alternative definitions
of $\Modasc$. Given a Cayley permutation $x$ of length~$n$, let
$$
\asctops(x)= \{(1,x_1)\}\cup \{(i, x_i): 1 < i \le n,\, x_{i-1} < x_i\}
$$
be the set of \emph{ascent tops} and their indices, including the
first element; further, let
$$
\nub(x) = \{(\min x^{-1}(j), j): 1\leq j\leq \max(x) \}
$$
be the set of \emph{leftmost copies} and their indices.
When there is no ambiguity, we will sometimes abuse notation and
simply write $x_i\in\nub(x)$ or $x_i\in\asctops(x)$. 
If $x_i\in\nub(x)$ and $x_i=a$, we say that $x_i$ is the
leftmost copy of $a$ in $x$; or, that $x_i$ is a leftmost copy
in $x$. It is easy to see~\cite{CC} that $x$ avoids~$\mathfrak{a}$
if and only if $\asctops(x)\subseteq\nub(x)$; similarly, $x$
avoids~$\mathfrak{b}$ if and only if $\asctops(x)\supseteq\nub(x)$.
The next proposition, which will be repeatedly used throughout the
whole paper, follows immediately.

\begin{proposition}\label{prop_modasc_prop}
We have
$$
\Modasc=\{x\in\Cay: \asctops(x)=\nub(x)\}.
$$
In particular, in a modified ascent sequence~$x$ all the ascent tops
have distinct values and $\max(x)=|\asctops(x)|$.
Furthermore, all the copies of $\max(x)$ are in consecutive
positions.
\end{proposition}

Finally, a recursive definition of $\Modasc$ goes as follows~\cite{CC}.
There is exactly one modified ascent sequence of length zero and one,
the empty word and the single letter word $1$, respectively.
For $n\geq 1$, every $y\in\Modasc_{n+1}$ is of one of two forms depending
on whether the last letter forms an ascent with the penultimate letter:
\begin{itemize}
\item $y=x_1\cdots x_n\;x_{n+1}$,\,with\, $1\le x_{n+1}\le x_n$, or
\item $y=\tilde{x}_1\cdots\tilde{x}_n\;x_{n+1}$,\,with\,
$x_n<x_{n+1}\le 1+\max(x_1\cdots x_n)$,
\end{itemize}
where $x_1\cdots x_n\in\Modasc_{n}$ and, for $i\in[n]$,
$$
\tilde{x}_i=
\begin{cases}
x_i & \text{if }x_i<x_{n+1}\\
x_i+1 & \text{if }x_i\ge x_{n+1}.
\end{cases}
$$
Less formally, each modified ascent sequence $x$ gives rise to
$\max(x)+1$ modified ascent sequences of length one more. These are
obtained by first inserting a new rightmost entry that is less than
or equal to $\max(x)+1$; and, secondly, if the newly added entry~$a$
is an ascent top, by increasing by one all the previous entries that
are greater than or equal to~$a$.

We wrap up this section with a remark. One of the main benefits of
working with modified ascent sequences is that they are Cayley
permutations. This is not the case of (plain) ascent sequences,
where the presence of gaps makes the study of patterns arguably
less natural. A rather awkward example is the following: there are
two ascent sequences of length five, namely $12123$ and $12124$, that
contain the length five pattern~$12123$.

\subsection{Primitive sequences}\label{section_primitive}

A \emph{flat step} in a modified ascent sequence~$x=x_1\cdots x_n$
consists of two consecutive equal entries $x_i=x_{i+1}$.
A modified ascent sequence is \emph{primitive}~\cite{DKRS} if it has
no flat steps, and we let $\Prim$ denote the set of primitive
modified ascent sequences. In the realm of Fishburn structures,
primitive (modified) ascent sequences are in bijection with binary
Fishburn matrices~\cite{DP}, $(\twoplustwo)$-free posets with no
indistinguishable elements~\cite{DKRS}, strictly-decreasing Fishburn trees~\cite{CC2}, and Fishburn permutations avoiding a bivincular
pattern of length two.
It is well known~\cite[Prop.~8]{DKRS} that any ascent sequence is
uniquely obtained from a primitive ascent sequence by inserting
flat steps in a suitable way. Clearly, e.g. by 
Proposition~\ref{prop_modasc_prop}, the same property holds for modified
ascent sequences too. For instance, the sequence
$$
x=1\;\underline{11}\;312\;\underline{222}\;421\;\underline{1}\in\Modasc
\quad\text{arises from}\quad
w=1312421\in\Prim
$$
by inserting the underlined flat steps. More interestingly,
if $y$ is a primitive pattern and $w\in\Modasc(y)$, then the insertion
of flat steps in $w$ does not create any occurrence of~$y$.
We state the enumerative consequences of this simple observation
in the following proposition.

\begin{proposition}\label{prop_prim_to_full}
Let $y\in\Prim$. Then, for $n\ge 1$,
\begin{equation}\label{eq_prim}
|\Modasc_n(y)|=\sum_{k=1}^n\binom{n-1}{k-1}|\Prim_k(y)|.
\end{equation}
\end{proposition}
\begin{proof}
Any $x\in\Modasc_n(y)$ is obtained uniquely from some $w\in\Prim_k(y)$,
with $1\le k\le n$, by inserting $n-k$ flat steps. Note that $w$ is
obtained by collapsing all the consecutive flat steps of $x$ to a single
entry. Since $x_1=1$ is fixed, there are $\binom{n-1}{k-1}$ positions
where the remaining $k-1$ entries of $w$ can be placed.
\end{proof}

A consequence of Proposition~\ref{prop_prim_to_full} is that
enumerating~$\Prim(y)$ is sufficient in order to count the whole
set~$\Modasc(y)$, when~$y$ is a primitive pattern.
We can also rephrase this result in terms of generating functions.
From now on, given a pattern~$y$, we let
$$
\Modasc_y(t)=\sum_{n\ge 0}|\Modasc_n(y)|t^n
\quad\text{and}\quad
\Prim_y(t)=\sum_{n\ge 0}|\Prim_n(y)|t^n
$$
be the {\ogf} (ordinary generating functions) of $\Modasc(y)$ and
$\Prim(y)$, respectively.
It is well known (see for instance Bernstein and Sloane~\cite{BS})
that if
$$
b_n=\sum_{k=0}^n\binom{n}{k}a_k
\quad\text{then}\quad
B(t)=\frac{1}{1-t}A\left(\frac{t}{1-t}\right),
$$
where $A(t)=\sum_{n\ge 0}a_nt^n$ and $B(t)=\sum_{n\ge 0}b_nt^n$.
By Proposition~\ref{prop_prim_to_full}, keeping track of the shift,
\begin{equation}\label{eq_prim_to_full}
\begin{aligned}
\frac{\Modasc_y(t)-1}{t}&=
\frac{1}{1-t}\left[\frac{\Prim_y(s)-1}{s}\right]_{|s=\frac{t}{1-t}}\\
\iff\Modasc_y(t)&=
1+\frac{t}{1-t}\left[\frac{\Prim_y(s)-1}{s}\right]_{|s=\frac{t}{1-t}}\\
&=\Prim_y\left(\frac{t}{1-t}\right).
\end{aligned}
\end{equation}
%

We end this section with a simple lemma.

\begin{lemma}\label{lemma_prim_stats}
If~$x$ is a primitive modified ascent sequence, then
$$
\wltrmax(x)=\ltrmax(x)
\quad\text{and}\quad
\wrtlmax(x)=\rtlmax(x).
$$
Furthermore, $\ltrmin(x)=\{x_1\}$.
\end{lemma}
\begin{proof}
It is clear that $\ltrmin(x)=\{x_1\}$ since $x_1=1$.
The inclusions $\wltrmax(x)\supseteq\ltrmax(x)$ and
$\wrtlmax(x)\supseteq\rtlmax(x)$ are trivial.
Let $x_i\in\wltrmax(x)$. Since $x$ is primitive,
we have $x_{i-1}<x_i$. Thus $x_i\in\asctops(x)=\nub(x)$ and
$x_i$ is a strict left-to-right maximum.
We have thus proved that $\wltrmax(x)=\ltrmax(x)$.
Next, let $x_i\in\wrtlmax(x)$. For a contradiction,
suppose that $x_i\notin\rtlmax(x)$; that is, there is some
$x_j=x_i$, with $j>i$. Note that $x_j\in\wrtlmax(x)$.
Further, it must be $j>i+1$ since $x$ is primitive.
Now, consider the entry $x_{j-1}$ preceding $x_j$.
If $x_{j-1}<x_j$, then we have a contradiction with the fact
that $x_j\in\asctops(x)=\nub(x)$ and $x_j=x_i$.
If $x_{j-1}=x_j$, then we have a flat step, which is forbidden.
Finally, if $x_{j-1}>x_j$, then $x_i\notin\wrtlmax(x)$,
which is once again a contradiction.
\end{proof}

\section{Standardization of $\Modasc$}\label{section_std}

A commonly used tool to reduce problems about multisets to sets
is given by the \emph{standardization} map, here denoted by~$\std$.
The name standardization is due to Stanley~\cite[Prop.~1.7.1]{St},
but the oldest reference we could find goes back to a classic paper
by Schensted~\cite{Sc} from 1961.
Let $x=x_1\cdots x_n$ be a Cayley permutation with $\max(x)=k$.
Let $a_i$ be the number of copies of~$i$ contained in $x$, for $i\in[k]$.
Then $\std(x)$ is the permutation obtained by replacing the $a_i$
copies of $i$ with
$$
a_1+\cdots+a_{i-1}+1,\;a_1+\cdots+a_{i-1}+2,\;\dots,\;a_1+\cdots+a_{i-1}+a_i,
$$
going from left to right. More informally, we replace the~$a_1$ copies
of~$1$ with the numbers~$1,2,\dots,a_1$, the~$a_2$ copies of~$2$
with~$a_1+1,a_1+2,\dots,a_1+a_2$, and so on. For instance,
we have $\std(312112341)=715236894$, where the~$1$s are replaced
by $1,2,3,4$, the~$2$s by $5,6$, the~$3$s by $7,8$, and the only~$4$
is replaced by $9$.
Some simple properties satisfied by the standardization map are listed
in the following three results, where~$x$ is a Cayley permutation of
length~$n$ and $p=\std(x)$. The easy proofs are omitted or just sketched.

\begin{lemma}\label{lemma_std_ascdes}
For each $i<j$,
$$
x_i\le x_j\iff p_i<p_j.
$$
In particular, standardization preserves (strict) descents and maps weak
ascents to ascents. Further, it maps flat steps $x_i=x_{i+1}$ to
ascents $p_{i+1}=p_i+1$ that are consecutive in value.
\end{lemma}

\begin{lemma}\label{lemma_std_nub}
Let $i<j$ such that $p_i=p_j+1$. Then $x_i\in\nub(x)$.
\end{lemma}
\begin{proof}
The assumption $p_i=p_j+1$ says that the $\std$ maps
``reads'' $x_{i}$ immediately after~$x_j$; since $i<j$,
the entry $x_{i}$ must be a leftmost copy in $x$.
\end{proof}

In the next lemma we abuse notation by writing
$\ltrmin(x)\subseteq\ltrmin(p)$ instead of
$\{i\in[n]:x_i\in\ltrmin(x)\}\subseteq \{i\in[n]:p_i\in\ltrmin(p)\}$
(the same in the other items).

\begin{lemma}\label{lemma_std_props}
We have:
$$
\def\arraystretch{1.33}
\begin{array}{lll}
(i)&  \ltrmin(x)=\ltrmin(p);& \wltrmin(x)\supseteq\ltrmin(p).\\
(ii)& \ltrmax(x)\subseteq\ltrmax(p);& \wltrmax(x)=\ltrmax(p).\\
(iii)&\rtlmin(x)\subseteq\rtlmin(p);& \wrtlmin(x)=\rtlmin(p).\\
(iv)& \rtlmax(x)=\rtlmax(p);& \wrtlmax(x)\supseteq\rtlmax(p).
\end{array}
$$
\end{lemma}

From now on, we let
$$
\omega=\begin{tikzpicture}[scale=0.4, baseline=20.5pt]
\fill[NE-lines] (1,0) rectangle (2,4);
\fill[NE-lines] (0,1) rectangle (4,2);
\draw [semithick] (0.001,0.001) grid (3.999,3.999);
\filldraw (1,3) circle (6pt);
\filldraw (2,2) circle (6pt);
\filldraw (3,1) circle (6pt);
\end{tikzpicture},\quad
\zeta=
\begin{tikzpicture}[scale=0.4, baseline=20.5pt]
\fill[NE-lines] (0,0.5) rectangle (1,1.5);
\draw [semithick] (1,0.5) -- (1,3.5);
\draw [semithick] (2,0.5) -- (2,3.5);
\draw [semithick] (0,1.5) -- (3,1.5);
\draw [semithick] (0,2.5) -- (3,2.5);
\filldraw (1,2.5) circle (6pt);
\filldraw (2,1.5) circle (6pt);
\end{tikzpicture},
\quad\text{and}\quad
\Omega=\Sym(\omega,\zeta).
$$
The reader who is familiar with generalized patterns will immediately
realize that~$\omega$ is in fact a bivincular pattern
$\omega=(321,\{1\},\{1\})$. Further, a permutation has~$1$ as the
leftmost entry if and only if it avoids~$\zeta$. Indeed, the 
set~$\Omega$ could be alternatively defined as the direct
sum $\Omega=1\oplus\Sym(\omega)$.

The main goal of this section is to prove that standardization
maps bijectively the set $\Prim$ of primitive modified ascent
sequences to $\Omega$. We shall proceed as follows.
First, we show that $\std(\Modasc)\subseteq\Omega$. Then, we prove
that every permutation in~$\Omega$ is the standardization of a primitive
modified ascent sequence. Since $\Prim\subseteq\Modasc$, we get
$$
\std(\Prim)\subseteq\std(\Modasc)\subseteq\Omega\subseteq\std(\Prim),
$$
from which $\std(\Prim)=\std(\Modasc)=\Omega$ is obtained immediately.
Finally, that $\std$ maps bijectively $\Prim$ to $\Omega$
follows since Parviainen~\cite[Section~5.4]{P} proved that primitive
(modified) ascent sequences and permutations in $\Omega$ are
equinumerous. Let us expand and clarify a bit on this last part.
Parviainen showed that $|\Omega_n|=|\Prim_n|$ by slightly tweaking
a bijection~$f$ claimed to be defined from ascent sequences
to Fishburn permutations. In fact, the map~$f$ should be defined
on modified ascent sequences~\cite{BMCDK}. Specifically, $f$ is a special
instance of the \emph{Burge transpose}~\cite{Bu,CC}. The Burge transpose
acts on biwords $(u,x)$ as follows. It flips the columns of $(u,x)$
upside down; then, it sorts the columns of the resulting biword
in ascending order with respect to the top entry, breaking ties by
sorting in descending order with respect to the bottom entry.
When $x\in\Modasc$ and $u=12\cdots n$, the bottom row of the
transpose of $(u,x)$ is the Fishburn permutation associated with~$x$.
If we break ties in the opposite way, i.e. by sorting in ascending
order with respect to the bottom entry, and we restrict the transpose
to primitive sequences, then we end up with the desired bijection
between~$\Prim_n$ and~$\Omega_n$.
For instance, the sequence~$x=1312\in\Prim$ is mapped to the
permutation~$1342\in\Omega$ since the transpose of the biword
$$
\binom{1\;2\;3\;4}{1\;3\;1\;2}
\quad\text{is}\quad
\binom{1\;1\;2\;3}{1\;3\;4\;2}.
$$
Note also that $\std(1312)=1423\neq 1342$.

\begin{proposition}
We have $\std(\Modasc)\subseteq\Omega$.
\end{proposition}
\begin{proof}
Let $x\in\Modasc$ and let $p=\std(x)$. For a contradiction, suppose
that $p\notin\Omega$. Note that $p_1=1$ since $x_1=1$.
Thus, since $\Omega=1\oplus\Sym(\omega)$, it must be that $p$ contains
an occurrence $p_ip_{i+1}p_j$ of $\omega$, where $p_i>p_{i+1}=p_j+1$
and $i+1<j$. By Lemma~\ref{lemma_std_ascdes}, we have $x_i>x_{i+1}$,
hence $x_{i+1}\notin\asctops(x)$.
On the other hand, we have $x_{i+1}\in\nub(x)$ by
Lemma~\ref{lemma_std_nub}. Thus $\nub(x)\neq\asctops(x)$, which is a
contradiction with $x$ being a modified ascent sequence.
\end{proof}

The proof that $\Omega\subseteq\std(\Prim)$ relies upon a geometric
decomposition of permutations in~$\Omega$ that stems from the next lemma.

\begin{lemma}\label{lemma_chains}
Let $p\in\Omega$. If $p_i\notin\asctops(p)$ and $p_j=p_i-1$,
then $j<i$.
\end{lemma}
\begin{proof}
If it were $j>i$, then $p_{i-1}p_ip_j$ would be an occurrence
of~$\omega$.
\end{proof}

Let $p\in\Omega$ and let $\asctops(p)=\{p_{k_1},\dots,p_{k_m}\}$,
where $m\ge 1$ and $p_{k_1}<p_{k_2}<\cdots<p_{k_m}$.
By Lemma~\ref{lemma_chains}, every entry that is not an ascent top
is located to the right of the next smaller entry in $p$. More
specifically, all the entries whose value is included bewteen two 
consecutive ascent tops, say $p_{k_i}$ and $p_{k_{i+1}}$,
appear in increasing order from left to right in $p$, and to the
right of $p_{k_i}$. This property allows us to partition $p$ in $m$ \emph{chains} of the form
$$
(p_{k_i},p_{k_i}+1,p_{k_i}+2,\dots,p_{k_{i+1}}-1),
$$
where the only ascent top in every chain is the first (and
smallest) element, and all the elements are consecutive in value and
appear in increasing order in $p$. An example of this construction
is depicted in Figure~\ref{figure_std_to_prim}.

\begin{proposition}\label{prop_perm_to_prim}
For each $p\in\Omega$, there is a primitive modified ascent
sequence $x$ such that $\std(x)=p$. In other words, we have
$\Omega\subseteq\std(\Prim)$.
\end{proposition}
\begin{proof}
Let $p\in\Omega$. We determine a modified ascent sequence~$x$ such
that~$\std(x)=p$. First, we define~$x$ with a geometric construction
illustrated in Figure~\ref{figure_std_to_prim}.
For each ascent top $p_i\in\asctops(p)$, draw a horizontal
half-line starting from $p_i$ and going to the right. Then, let
each other entry of~$p$ fall under the action of gravity until
it hits one of the horizontal lines defined before. Finally, rescale
the resulting word (by ignoring eventual vertical gaps created at
the previous step) in order to obtain a Cayley permutation~$x$.
More formally, let $y$ be the string obtained from $p$ by letting
$$
y_i=\max(U_i),
\quad\text{where}\quad
U_i=\{p_j:\; j<i,\;p_j<p_i,\;p_j\in\asctops(p)\},
$$
for each $p_i\notin\asctops(p)$, and $y_i=p_i$ otherwise. Finally,
let $x$ be the only Cayley permutation order isomorphic to $y$.
Note that every $p_i\notin\asctops(p)$ necessarily hits some half-line
since there is a half-line starting from $p_1=1$; equivalently,
the set $U_i$ is not empty since $p_1=1\in\asctops(p)$.
The construction of~$x$ can be alternatively described in terms of
the chains of $p$ (defined just before this proposition):
all the entries in the same chain fall at the same level as the
leftmost element of the chain, which is the only ascent top of the
chain, as well as its smallest entry.
The equivalence of the two definitions is omitted.
To complete the proof, we need to show that $x\in\Modasc$, $x$
contains no flat steps, and $\std(x)=p$.
We just sketch the proof of these claims, leaving some technicalities
to the reader.
\begin{itemize}
\item To see that $x\in\Modasc$, observe that $p_i\in\asctops(p)$
if and only if $x_i\in\asctops(x)$. Now, if
$p_i\in\asctops(p)$, then $x_i\in\nub(x)$ as well. On the other hand,
if $p_i\notin\asctops(p)$, then~$p_i$ falls at the same level as
some $p_j\in\asctops(p)$, with $j<i$, and thus $x_i\notin\nub(x)$.
Hence, we have $\asctops(x)=\nub(x)$, and $x\in\Modasc$ follows.
Note that we did not use that $p$ avoids $\omega$ here.
\item Next, we show that the avoidance of $\omega$ guarantees
that $x$ contains no flat steps. 
For a contradiction, suppose that $x_i=x_{i+1}$ is a flat step in $x$.
Note that it must be $p_{i+1}\notin\asctops(p)$, or else $p_{i+1}$
would not fall. Thus we have $p_i>p_{i+1}$ and, since $x_i=x_{i+1}$,
we have $p_i\notin\asctops(p)$ as well. Since $p_i$ and $p_{i+1}$ fall
at the same level, they must belong to the same chain. But this
is impossible since entries in a chain appear in increasing
order in $p$ and $p_i>p_{i+1}$.
\item To see that $\std(x)=p$, observe that the entries
of $x$ that are equal to $1$ correspond to the chain of $p$ whose
smallest entry is $p_1=1$. As observed after Lemma~\ref{lemma_chains},
such chain is $(p_1,2,3,\dots,\ell_1)$, for some $\ell_1\ge 1$.
Further, standardization sets the $i$th copy of $1$ equal to $i$,
matching the desired value of each entry in $p$.
The same argument holds for the remaining chains of $p$, and our
claim follows.
\end{itemize}
This concludes the proof.
\end{proof}

\begin{figure}
\centering
\begin{tikzpicture}[scale=0.4, baseline=20pt]
	\draw[](0,0)--(2,11)--(4,12)--(6,14)--(8,2)--(10,5)--(12,3)--(14,7)--(16,13)--(18,8)--(20,6)--(22,4)--(24,10)--(26,9);
	\draw[dotted] (0,0)--(26.5,0);
	\draw[dotted] (2,11)--(26.5,11);
	\draw[dotted] (4,12)--(26.5,12);
	\draw[dotted] (6,14)--(26.5,14);
	\draw[dotted] (10,5)--(26.5,5);
	\draw[dotted] (14,7)--(26.5,7);
	\draw[dotted] (16,13)--(26.5,13);
	\draw[dotted] (24,10)--(26.5,10);
	\node[above left] at (0,0){$1$};
	\filldraw (0,0) circle (5pt);
	\node[above left] at (2,11){$11$};
	\filldraw (2,11) circle (5pt);
	\node[above left] at (4,12){$12$};
	\filldraw (4,12) circle (5pt);
	\node[above] at (6,14){$14$};
	\filldraw (6,14) circle (5pt);
	\node[left] at (8,2){$2$};
	\filldraw (8,2) circle (5pt);
	\node[above] at (10,5){$5$};
	\filldraw (10,5) circle (5pt);
	\node[right] at (12,3){$3$};
	\filldraw (12,3) circle (5pt);
	\node[left] at (14,7){$7$};
	\filldraw (14,7) circle (5pt);
	\node[above] at (16,13){$13$};
	\filldraw (16,13) circle (5pt);
	\node[above right] at (18,8){$8$};
	\filldraw (18,8) circle (5pt);
	\node[above right] at (20,6){$6$};
	\filldraw (20,6) circle (5pt);
	\node[right] at (22,4){$4$};
	\filldraw (22,4) circle (5pt);
	\node[above] at (24,10){$10$};
	\filldraw (24,10) circle (5pt);
	\node[above] at (26,9){$9$};
	\filldraw (26,9) circle (5pt);
	\draw (8,0) circle (5pt);
	\draw (12,0) circle (5pt);
	\draw (18,7) circle (5pt);
	\draw (20,5) circle (5pt);
	\draw (22,0) circle (5pt);
	\draw (26,7) circle (5pt);
	\draw[<->,thick,blue] (8,1.9)--(8,0.1);
	\draw[<->,thick,blue] (12,2.9)--(12,0.1);
	\draw[<->,thick,blue] (18,7.9)--(18,7.1);
	\draw[<->,thick,blue] (20,5.9)--(20,5.1);
	\draw[<->,thick,blue] (22,3.9)--(22,0.1);
	\draw[<->,thick,blue] (26,8.9)--(26,7.1);
\end{tikzpicture}
\caption{The primitive modified ascent sequence $x=15681213732143$
associated with the permutation $p=1,11,12,14,2,5,3,7,13,8,6,4,10,9$
in $\Omega$. Note that $\std(x)=p$. The chains of~$p$ of length two
or more are $(1,2,3,4)$, $(5,6)$, $(7,8,9)$.}\label{figure_std_to_prim}
\end{figure}
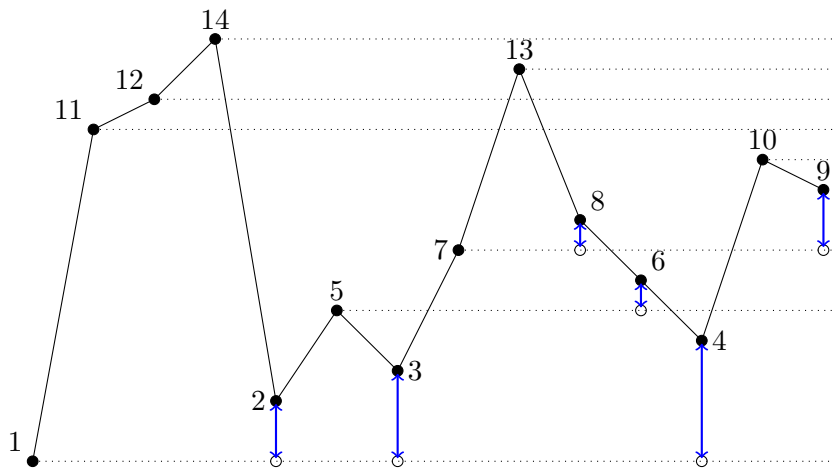

\begin{corollary}
Standardization is a size-preserving bijection from $\Prim$ to $\Omega$.
\end{corollary}

Schensted~\cite{Sc} observed that the decreasing subsequences of~$x$
and~$p=\std(x)$ are in one-to-one correspondence, while the increasing
subsequences of~$p$ are in one-to-one correspondence with the weakly increasing subsequences of~$x$. Roughly speaking, the reason is that
the behavior of the standaridization map on any subsequence of~$x$ is
not affected by the remaining entries of~$x$.
Specifically, if $x_{i_1}\cdots x_{i_k}$ is an occurrence of~$y$
in~$x$, then $p_{i_1}\cdots p_{i_k}$ is an occurrence of~$\std(y)$
in~$p$. Conversely, if $p_{i_1}\cdots p_{i_k}\simeq q$, then
$\std(x_{i_1}\cdots x_{i_k})=q$ as well. The following theorem of
transport of patterns from~$\Omega$ to~$\Prim$ is obtained immediately.

\begin{theorem}\label{thm_transport}
Given $p\in\Sym$, let $[p]=\{x\in\Cay:\std(x)=p\}$. Then standardization
is a size-preserving bijection from $\Prim[p]$ to $\Omega(p)$.
\end{theorem}

Theorem~\ref{thm_transport} is analogous to the transport
theorem between Fishburn permutations and modified ascent
sequences~\cite[Thm.~5.1]{CC}. Standardization plays the role
of the Burge transpose, and the set $[p]$ replaces the Fishburn
basis. As we will see later, by pairing Theorem~\ref{thm_transport}
with Proposition~\ref{prop_prim_to_full} we are sometimes able to
rephrase the original problem of counting $\Modasc(y)$ in terms of
permutations, making our task much easier. Examples where this
approach is fruitful can be found in Section~\ref{section_213_231_321}
and Section~\ref{section_221_2321}.

\section{Easy patterns}\label{section_easy}

As a warm up for the next sections, we solve some simple
patterns of short length.

\subsection{Patterns of length two}\label{sec_patt_len2}

The only modified ascent sequence of length~$n$ that avoids~$11$ is
the strictly increasing sequence $x=12\dots n$. Similarly,
there is only one sequence that avoids~$12$, namely the sequence
containing all ones $x=11\cdots 1$.

A modified ascent sequence avoids~$21$ if and only if it is a
weakly increasing Cayley permutation~\cite{CC}, and the number
of such sequences of length~$n$ is $2^{n-1}$. Further~\cite{Ce2},
we have $\Modasc(21)=\Modasc(121)$.

\subsection{Pattern~$112$}\label{sec_112}

Let $x\in\Modasc(112)$. Then $x=1y1^{k_1}$, for some $k_1\ge 0$,
where each entry in $y$ is strictly greater than~$1$.
Indeed, no entry greater than or equal to two is allowed to appear
to the right of the second copy of~$1$.
Further, by Proposition~\ref{prop_modasc_prop}, the subsequence~$y$
is order isomorphic to some $\tilde{y}\in\Modasc(112)$; namely,
$y$ is obtained by increasing by one each entry of $\tilde{y}$.
Iterating the same argument on $y$ yields a ``left pyramid'' structure:
$$
x=12\cdots m m^{k_m} \cdots 2^{k_2}1^{k_1},
\qquad
\begin{tikzpicture}[scale=0.25, baseline=20pt]
\draw[](0,0)--(2.5,5)--(3.5,5);
\draw[](4.5,4)--(5.5,4);
\draw[](6.5,3)--(7.5,3);
\draw[](8.5,2)--(9.5,2);
\draw[](10.5,1)--(11.5,1);
\draw[](12.5,0)--(13.5,0);
\draw[dotted](3.5,5)--(4.5,4);
\draw[dotted](5.5,4)--(6.5,3);
\draw[dotted](7.5,3)--(8.5,2);
\draw[dotted](9.5,2)--(10.5,1);
\draw[dotted](11.5,1)--(12.5,0);
\filldraw (0,0) circle (5pt);
\filldraw (0.5,1) circle (5pt);
\filldraw (1,2) circle (5pt);
\filldraw (1.5,3) circle (5pt);
\filldraw (2,4) circle (5pt);
\filldraw (2.5,5) circle (5pt);
\end{tikzpicture}
$$
where $m=\max(x)$ and $k_i\ge 0$, for $i=1,\dots,m$. Therefore,
any $x\in\Modasc_n(112)$ is uniquely determined by a tuple
$(k_1+1,k_2+1,\dots,k_m+1)$ recording the multiplicity of its values;
that is, by a composition of~$n$ (with $m=\max(x)$ parts).
Finally, the number of compositions of~$n$ is well known to be
equal to $2^{n-1}$.

With a little more effort, we can enumerate $\Prim(112)$.
A $112$-avoiding modified ascent sequence as above is primitive
if and only if $k_m=0$ and $k_i\in\{0,1\}$ for each $i<m$.
In other words, by ignoring the last entry $k_m+1=1$ in the tuple
$(k_1+1,k_2+1,\dots,k_m+1)$, we obtain a composition of $n-1$ with
no parts greater than two. A quick look in the OEIS~\cite{Sl} reveals
that the number of such compositions of $n-1$ is given by the $n$th
Fibonacci number.

Computing the number of primitive sequences in the cases discussed so
far is a fairly easy task. The interested reader is invited to
check Table~\ref{table_patts} to see the resulting sequences.

\subsection{Pattern~$122$}\label{section_122}

Let $x\in\Modasc_n(122)$. Since $x_1=1$, every integer between~$2$
and $\max(x)$ appears exactly once in~$x$.
Furthermore, all the entries between two copies of~$1$ appear in
increasing order due to the equality~$\nub(x)=\asctops(x)$.
In other words, if $x$ contains~$k$ copies of~$1$, then $x$
decomposes as
$$
x=1B_1\;1B_2\;\dots\;1B_k,
$$
where entries in each block $B_i$ are greater than or equal to~$2$,
and~$B_i$ is strictly increasing (possibly empty). Thus,
$$
|\{x\in\Modasc_n(122):\text{$x$ contains $k$ copies of $1$}\}|=
k^{n-k}.
$$
Indeed, a sequence~$x$ as above is determined by choosing,
for each of the $n-k$ entries greater than~$1$, the index
$i\in\{1,2,\dots,k\}$ of its block~$B_i$. Summing over~$k$, we get
$$
|\Modasc_n(122)|=\sum_{k=1}^nk^{n-k}.
$$
According to A026898~\cite{Sl}, the size of $\Modasc_n(122)$ is equal
to the number of set partitions of~$[n]$ whose minima of blocks form
an interval. A simple bijective proof goes as follows.
Given $x\in\Modasc_n(122)$, insert a block separator before every copy
of $1$ (ignoring the leftmost one), and compute $\std(x)$ as usual.
The result is a set partition whose minima of blocks correspond
to the copies of $1$ in $x$ . For instance, if $x=134112561$:
$$
x=134|1|1256|1
\;\longmapsto\;
\std(x)=167|2|3589|4=\{1,6,7\}\{2\}\{3,5,8,9\}\{4\}.
$$
We were not able to find a reference for the {\ogf} given
in A026898, and we wish to fill this gap below.
Recall that
$$
\Modasc_{122}(t)=\sum_{n\ge 0}|\Modasc_n(122)|t^n
$$
denotes the {\ogf} of $122$-avoiding modified ascent sequences.
An {\ogf} for sequences in $\Modasc_{n}(122)$ that contain exactly~$k$
copies of~$1$ is
$$
t^k\sum_{m\ge 0}k^mt^m=\frac{t^k}{1-kt},
$$
where $n=m+k$. Summing over~$k$, we obtain
$$
\Modasc_{122}(t)=\sum_{k\ge 0}\frac{t^k}{1-kt}.
$$
Finally, an {\ogf} for the sequence A026898, which is shifted by
one position compared to $\Modasc_{122}(t)$, is
$$
\frac{1}{t}(\Modasc_{122}(t)-1)=\sum_{k\ge 0}\frac{t^k}{1-(k+1)t},
$$
which matches the one given in the OEIS.

To end this section, we wish to enumerate $\Prim(122)$,
something we will use in Section~\ref{section_1232}.
Let $n\ge 1$ and let $x\in\Prim_n(122)$. Once again, we shall
decompose $x$ by highlighting the copies of~$1$ it contains. The
only difference compared to the general case, is that only the last
block is allowed to be empty since $x$ is primitive (and any other
empty block would result in two consecutive copies of~$1$). Thus,
if $x\in\Prim_n(122)$ contains $k$ copies of~$1$, we have either
$$
x=1B_1\;\dots\;1B_{k-1}\;1B_k
\quad\text{or}\quad
x=1B_1\;\dots\;1B_{k-1}\;1,
$$
according to whether or not $B_k$ is empty.
Clearly, the former are (in bijection with) ordered set partions
of size $n-k$ with $k$ blocks, which are counted by $k!S(n-k,k)$;
the latter are ordered set partitions of size $n-k$ with $k-1$ blocks,
counted by $(k-1)!S(n-k,k-1)$. Here, we denote by $S(n,k)$ the $(n,k)$th
Stirling number of the second kind. Finally, for $n\ge 1$ we obtain
\begin{align*}
|\Prim_n(122)|&=\sum_{k\ge 1}\bigl[k!S(n-k,k)+(k-1)!S(n-k,k-1)\bigr]\\
&=\sum_{k\ge 1}(k-1)!\bigl(kS(n-k,k)+S(n-k,k-1)\bigr)\\
&=\sum_{k\ge 1}(k-1)!S(n-k+1,k).
\end{align*}
For the rest of this section, let
\begin{align*}
F(t)&=\sum_{n\ge 0}\sum_{k\ge 0}k!S(n-k,k)t^{n},
\shortintertext{so that}
\Prim_{122}(t)&=1+\sum_{n\ge 1}|\Prim_n(122)|t^n\\
&=1+\sum_{n\ge 1}\sum_{k\ge 0}k!S(n-k,k)t^n+
\sum_{n\ge 1}\sum_{k\ge 1}(k-1)!S(n-k,k-1)t^n\\
&=F(t)+\sum_{n\ge 1}\sum_{j\ge 0}j!S(n-j-1,j)t^{n}\\
&=F(t)+t\sum_{m\ge 0}\sum_{j\ge 0}j!S(m-j,j)t^{m}\\
&=(1+t)F(t).
\end{align*}
A shift by one position of $\Prim_{122}(t)$ is recorded as
A229046. Cao {\etal}~\cite{CJL} showed that its
$n$-th term---i.e. $|\Prim_{n+1}(122)|$---is equal to the number of
inversion sequences of length~$n$ avoiding the triple of binary
relations~$(-,-,=)$; or, equivalently, avoiding the patterns~$111$,
$121$ and~$212$. A bijection between the two structures remains
to be found. Similarly, a shift of $F(t)$ gives A105795.
Each of these two entries in the OEIS contains (at least) an {\ogf}
for the corresponding sequence, but we could not find any formal proof.
We bridge this gap below, starting from $F(t)$.
Stanley~\cite[Eq.~(1.94)]{St} proved the following two equations
involving the Stirling numbers of the second kind:
\begin{align*}
k!S(n,k)&=\sum_{i\ge 0}(-1)^{k-i}\binom{k}{i}i^n; &(1.94)(a)\\
\sum_{m\ge 0}S(m,k)t^m&=\frac{t^k}{(1-t)(1-2t)\cdots(1-kt)}. &(1.94)(c)
\end{align*}
Now,
\begin{align*}
F(t)&=\sum_{n\ge 0}\sum_{k\ge 0}k!S(n-k,k)t^{n}\\
&=\sum_{m\ge 0}\sum_{k\ge 0}k!S(m,k)t^{m+k} & m=n-k\\
&=\sum_{k\ge 0}k!t^k\frac{t^k}{(1-t)(1-2t)\cdots(1-kt)} & \text{By }(1.94)(c)\\
&=\sum_{k\ge 0}\prod_{j=1}^k\frac{jt^2}{1-jt}.
\end{align*}
Alternatively,
\begin{align*}
F(t)&=\sum_{m\ge 0}\sum_{k\ge 0}k!S(m,k)t^{m+k}\\
&=\sum_{m\ge 0}\sum_{k\ge 0}\sum_{i\ge 0}(-1)^{k-i}\binom{k}{i}i^mt^{m+k}
& \text{By }(1.94)(a)\\
&=\sum_{i\ge 0}\sum_{k\ge 0}(-1)^{k-i}\binom{k}{i}t^k\left(\sum_{m\ge 0}i^mt^m\right)\\
&=\sum_{i\ge 0}\frac{t^i}{1-it}\left(\sum_{k\ge 0}(-1)^{k-i}\binom{k}{i}t^{k-i}\right)\\
&=\sum_{i\ge 0}\frac{t^i}{(1-it)(1+t)^{i+1}},
\end{align*}
where the last step follows from the binomial theorem:
\begin{align*}
(1+t)^{-i-1}&=\sum_{j\ge 0}\binom{-i-1}{j}t^j\\
&=\sum_{j\ge 0}(-1)^j\binom{i+1+j-1}{j}t^j\\
&=\sum_{k\ge 0}(-1)^{k-i}\binom{k}{k-i}t^{k-i} & k=i+j\\
&=\sum_{k\ge 0}(-1)^{k-i}\binom{k}{i}t^{k-i}.
\end{align*}
We have thus proved the following proposition.
\begin{proposition}\label{prop_ogf_122}
Let $F(t)=\sum_{n\ge 0}\sum_{k\ge 0}k!S(n-k,k)t^{n}$. Then
$$
F(t)=\sum_{k\ge 0}\prod_{j=1}^k\frac{jt^2}{1-jt}
=\sum_{i\ge 0}\frac{t^i}{(1-it)(1+t)^{i+1}}.
$$
\end{proposition}
Two {\ogf}s for $\Prim_{122}(t)$ are obtained immediately
as $\Prim_{122}(t)=(1+t)F(t)$. An {\ogf} for A229046 is
$$
G(t)=\frac{1}{t}(\Prim_{122}(t)-1)=F(t)+\frac{1}{t}(F(t)-1).
$$
Using Proposition~\ref{prop_ogf_122}, we compute
$$
G(t)=\sum_{k\ge 0}\frac{1}{1-(k+1)t}\prod_{j=1}^k\frac{jt^2}{1-jt},
$$
which agrees with the {\ogf} given in A229046.

\section{Primitive patterns}\label{section_prim_patts}

This whole section is devoted to the solution of primitive patterns.

\subsection{Pattern~$1232$}\label{section_1232}

We start by enumerating~$\Modasc(1232)$. The key is
the following lemma.

\begin{lemma}\label{lemma_1232}
For each $n\ge 0$,
$$
\Prim_n(122)=\Prim_n(1232).
$$
\end{lemma}
\begin{proof}
Clearly, if~$x$ avoids~$122$ then it avoids~$1232$ too. The
inclusion $\Prim(122)\subseteq\Prim(1232)$ follows.
Conversely, if $x\in\Prim$ contains an occurrence $x_ix_jx_k$ of~$122$,
then $x_k\notin\nub(x)=\asctops(x)$ and $x_ix_jx_{k-1}x_k$ is
an occurrence of $1232$.
\end{proof}

\begin{proposition}
For $n\ge 1$,
$$
|\Modasc_n(1232)|=\sum_{k=1}^n\binom{n-1}{k-1}\sum_{j=1}^k(j-1)!S(k-j+1,j).
$$
Furthermore,
$$
\Modasc_{1232}(t)
=\sum_{i\ge 0}\frac{t^i(1-t)}{1-(i+1)t}.
$$
\end{proposition}
\begin{proof}
The first statement follows from Proposition~\ref{prop_prim_to_full},
Lemma~\ref{lemma_1232}, and the equality
$|\Prim_k(122)|=\sum_{j=1}^k(j-1)!S(k-j+1,j)$, proved in
Section~\ref{section_122}.
As observed below Proposition~\ref{prop_ogf_122}, two expressions for
the {\ogf} of $\Prim(122)$ can be obtained as $\Prim_{122}(t)=(1+t)F(t)$.
Since $\Prim_{122}(t)=\Prim_{1232}(t)$, the second statement follows---with
a little bit of additional work---by setting $y=1232$ in
Equation~\eqref{eq_prim_to_full}.
\end{proof}

A shift by one position of $\Modasc_{1232}(t)$ is recorded as
A047970~\cite{Sl}.

\subsection{Patterns $213$, $231$ and $321$}\label{section_213_231_321}

We solve the patterns~$y\in\{213,231,321\}$ with the
machinery of Theorem~\ref{thm_transport}.
First, we show that in each of these cases standardization maps
bijectively $\Prim(y)$ to $\Omega(y)$. Then, we count~$\Omega(y)$
and use Proposition~\ref{prop_prim_to_full} to recover the
full enumeration of~$\Modasc(y)$. Let us start with a simple lemma.

\begin{lemma}\label{lemma_213_231}
We have
$$
\Prim(212,213)=\Prim(213)
\quad\text{and}\quad
\Prim(221,231)=\Prim(231).
$$
\end{lemma}
\begin{proof}
Showing that $\Prim(213)$ is contained in $\Prim(212,213)$
is sufficient to prove the first equality. Let $x\in\Prim(213)$. 
For a contradiction, suppose that $x$ contains~$212$ and let
$x_ix_jx_k$ be an occurrence of~$212$ in $x$. Note that
$x_k\notin\nub(x)=\asctops(x)$. Since~$x$ is primitive, it must be
$x_{k-1}>x_k$. Hence $x_ix_jx_{k-1}$ is an occurrence of $213$,
which is impossible. The second equality can be proved similarly.
If $x_ix_jx_k\simeq 221$, then it must be $x_{j-1}>x_j$, and
$x_ix_{j-1}x_k\simeq 231$.
\end{proof}

To prove the next result, we combine the previous lemma
with the transport theorem. Recall from Theorem~\ref{thm_transport}
that standardization maps bijectively $\Prim[p]$ to $\Omega(p)$,
where $[p]=\{x\in\Cay:\std(x)=p\}$ and $\Omega=1\oplus\Sym(\omega)$.

\begin{corollary}\label{cor_213_231_321}
For $n\ge 1$, standardization maps bijectively:
\begin{align*}
\Prim_n(213)&\longrightarrow\Omega_n(213);\\
\Prim_n(231)&\longrightarrow\Omega_n(231);\\
\Prim_n(321)&\longrightarrow 1\oplus\Sym_{n-1}(321).
\end{align*}
\end{corollary}
\begin{proof}
Observe that $[213]=\{212,213\}$ and $[231]=\{221,231\}$.
By Lemma~\ref{lemma_213_231}, $\Prim(213)=\Prim[213]$
and $\Prim(231)=\Prim[231]$. The first two items follow
immediately by  Theorem~\ref{thm_transport}.
The last item follows as well since $[321]=\{321\}$ and $321$ is
the classical pattern underlying~$\omega$.
\end{proof}

Now, it is easy to prove that $\Modasc(321)$ is counted by the
binomial transform of the Catalan numbers, shifted by one position
(A007317 in the OEIS~\cite{Sl}).

\begin{proposition}
For $n\ge 1$, we have
$$
|\Modasc_n(321)|=\sum_{j=0}^{n-1}\binom{n-1}{j}c_{j},
$$
where $c_j=\frac{1}{j+1}\binom{2j}{j}$ is the $j$th Catalan number.
\end{proposition}
\begin{proof}
We have:
\begin{align*}
|\Modasc_n(321)|&=\sum_{k=1}^n\binom{n-1}{k-1}|\Prim_k(321)|
& \text{by Proposition~\ref{prop_prim_to_full}}\\
&=\sum_{k=1}^n\binom{n-1}{k-1}|1\oplus\Sym_{k-1}(321)|
& \text{by Corollary~\ref{cor_213_231_321}}\\
&=\sum_{k=1}^n\binom{n-1}{k-1}c_{k-1}
& \text{since $|\Sym_{k-1}(321)|=c_{k-1}$}\\
&=\sum_{j=0}^{n-1}\binom{n-1}{j}c_j.
\end{align*}
\end{proof}

\begin{remark*}
The set $\RGF(321)$ of restricted growth functions avoiding~$321$
is equinumerous~\cite{CCFS} with~$\Modasc(321)$. Note that $\RGF$
encodes set partitions in the same way as~$\Cay$ encodes ordered
set partitions (and $\Modasc\subseteq\Cay$). Is there any other
example of Wilf-equivalence between pattern-avoiding $\RGF$s
and modified ascent sequences?
\end{remark*}

Let us take care of the patterns $213$ and $231$ next.
For $n\ge 0$, denote by~$m_n$ the $n$th \emph{Motzkin number} (see
also A001006~\cite{Sl}).

\begin{proposition}\label{prop_motzkin}
For $y\in\{213,231\}$ and $n\ge 0$, we have
$$
|\Sym_n(\omega,y)|=m_n.
$$
\end{proposition}
\begin{proof}
Let $M(t)=\sum_{n\ge 0}|\Sym_n(\omega,y)|t^n$.
We show that $M=M(t)$ satisfies
\begin{equation}\label{eq_motzkin}
M=1+tM+t^2M^2,
\end{equation}
a combinatorial equation defining the Motzkin numbers.
Let us start from the pattern~$y=213$. Let $p\in\Sym(\omega,213)$.
If $p$ is not the empty permutation, then~$p$ decomposes as $p=L1R$,
where $L$ and $R$ are possibily empty. Since $p$ avoids $213$, we
have $L>R$, i.e. each entry in the prefix~$L$ is greater than each
entry in the suffix~$R$. Also, each of~$L$ and $R$ is (order
isomorphic to) a permutation avoiding $\omega$ and $213$.
Now, there are exactly two possibilites:
\begin{itemize}
\item $L=\emptyset$. Then $p=1R$, which gives the $tM$ term in
Equation~\eqref{eq_motzkin}.
\item $L\neq\emptyset$. In this case, the smallest entry of $L$ is
forced to be in the leftmost position of~$L$; indeed, let
$p_i=\min(L)$ and let $j$ be such that $p_j=p_i-1$. Note that
either $p_j=1$ or $p_j\in R$. In any case, it must be $j>i$. Thus,
if it were $i\ge 2$, then we would have an occurrence
$p_{i-1}p_ip_j$ of $\omega$ in $p$, which is impossible.
We have thus showed that the position of the smallest entry
of $L$ is forced. On the other hand, the remaining entries of $L$
(and $R$) are allowed to form any $(\omega,213)$-avoiding
permutation. This contributes with the $t^2M^2$ term in
Equation~\eqref{eq_motzkin}.
\end{itemize}
In the end,
$$
M\;=\;\underbrace{1}_{\text{empty}}\;+\;
\underbrace{t\cdot M}_{L=\emptyset}\;+\;
\underbrace{t^2\cdot M^2}_{L\neq\emptyset}.
$$
The equation for $y=231$ is obtained similarly.
Any $p\in\Sym(\omega,231)$ decomposes as $p=LnR$, with $L<R$,
and the smallest entry of $R$ is forced to be in the leftmost
position of $R$ by (the avoidance of) $\omega$.
\end{proof}

\begin{corollary}
Let $y\in\{213,231\}$. Then $|\Modasc_n(y)|$ is equal to the $n$th
Catalan number.
\end{corollary}
\begin{proof}
The case $n=0$ is trivial. For $n\ge 1$,
\begin{align*}
|\Modasc_n(y)|&=\sum_{k=1}^n\binom{n-1}{k-1}|\Prim_k(y)|
& \text{by Proposition~\ref{prop_prim_to_full}}\\
&=\sum_{k=1}^n\binom{n-1}{k-1}|\Omega_k(y)|
& \text{by Corollary~\ref{cor_213_231_321}}\\
&=\sum_{k=1}^n\binom{n-1}{k-1}|1\oplus\Sym_{k-1}(\omega,y)|\\
&=\sum_{k=1}^n\binom{n-1}{k-1}m_{k-1}
& \text{by Proposition~\ref{prop_motzkin}}\\
&=c_n,
\end{align*}
where the last one is a well known equality relating the Motzkin
and the Catalan numbers (see Donaghey~\cite[Eq.~(2)]{Do}).
\end{proof}

\subsection{Patterns~$123$ and $1234$}\label{section_123}

The enumeration of modified ascent sequences avoiding~$y\in\{123,1234\}$
can be obtained as a consequence of the transport of patterns
developed by Claesson and the current author~\cite{CC}. Indeed, for
every $k\ge 1$, the Burge transpose maps bijectively
$\Modasc_n(12\cdots k)$ to $\F_n(12\cdots k)$, where $\F(12\cdots k)$
is the set of Fishburn permutations~\cite{BMCDK} avoiding $12\cdots k$;
further, Gil and Weiner~\cite{GW} proved that
$$
|\F_n(123)|=2^{n-1}
\quad\text{and}\quad
|\F_n(1234)|=c_n.
$$
An alternative and arguably more direct approach for $y=123$
consists in counting~$\Prim(123)$ and using our favourite
Proposition~\ref{prop_prim_to_full}. Indeed, we have
$$
\Prim(123)=\lbrace\epsilon,1,12,121,1312,13121,141312,1413121,\dots\rbrace,
$$
where $\epsilon$ denotes the empty sequence. In other words,
there is only one primitive, $123$-avoiding modified
ascent sequence of length~$n$; namely, the sequence

\begin{center}
\begin{tabular}{ll}
$1k1(k-1)1(k-2)1\cdots 121$ & \text{if $n=2k-1$ is odd};\\
$1k1(k-1)1(k-2)1\cdots 12$  & \text{if $n=2(k-1)$ is even}.
\end{tabular}
\end{center}

Finally,
\begin{align*}
|\Modasc_n(123)|=\sum_{k=1}^n\binom{n-1}{k-1}|\Prim_k(123)|
=\sum_{k=1}^n\binom{n-1}{k-1} = 2^{n-1}.
\end{align*}

Note also that (again by Proposition~\ref{prop_prim_to_full})
primitive modified ascent sequences avoiding $1234$ are enumerated
by the Motzkin numbers

\subsection{Pattern $132$}\label{section_132}

We prove that $132$-avoiding modified ascent sequences are
counted by the odd Fibonacci numbers (A001519~\cite{Sl}).
As usual, let us count the primitive sequences first.

Let $n\ge 1$ and let $x\in\Prim_n(132)$. Recall from
Proposition~\ref{prop_modasc_prop} that all the copies of $\max(x)$
are in consecutive positions. Since $x$ is primitive, it contains
only one copy of its maximum value. Let $m\in[n]$ denote the index
of the only entry $x_m=\max(x)$. We show that either $m=n-1$
or $m=n$. There is nothing to prove if $n\le 3$. Otherwise, let $n>3$. For a contradiction, assume $m\le n-2$. Since $x$ is primitive,
at least one of the last two entries, say $x_i$, $i\in\{n-1,n\}$,
is not equal to~$1$; hence~$x_1x_mx_i$ is an occurrence of $132$,
which is impossible. Consequently, any $x\in\Prim(132)$ falls in
exactly one of the following two cases:
\begin{itemize}
\item $m=n$. In this case, $x_1\cdots x_{n-1}\in\Prim_{n-1}(123)$
and $x_n=\max(x_1\cdots x_{n-1})+1$.
\item $m=n-1$. In this case, it must be $x_n=1$, or else we would
have~$x_1x_mx_n\Prim_{122}\simeq 132$. Specifically, we have
$x_1\cdots x_{n-2}\in\Prim_{n-2}(123)$,
$x_{n-1}=\max(x_1\cdots x_{n-2})+1$, and $x_n=1$.
\end{itemize}
Conversely, it is easy to see that inserting a suffix $m$ or $m1$
to any $y\in\Prim(132)$, where $m=\max(y)+1$, yields a primitive,
$132$-avoiding modified ascent sequence. Therefore,
$$
|\Prim_n(132)|=\underbrace{|\Prim_{n-1}(132)|}_{m=n}
+\underbrace{|\Prim_{n-2}(132)|}_{m=n-1}.
$$
Since $|\Prim_0(132)|=|\Prim_1(132)|=1$, it follows that
$|\Prim_n(132)|$ is equal to the $n$th Fibonacci number~$f_n$.
In the end, a well known formula for the odd-indexed Fibonacci numbers
gives
\begin{align*}
|\Modasc_n(132)|&=\sum_{k=1}^n\binom{n-1}{k-1}|\Prim_k(132)|\\
&=\sum_{k=1}^n\binom{n-1}{k-1}f_k = f_{2n-1}.
\end{align*}

\subsection{Pattern~$312$}\label{section_312}

In this section, we give a bijection between $\Prim_{n+1}(312)$
and the set of Dyck paths of semilength~$n$ that avoid the
(consecutive) subpath~$\dudu$. Sapounakis {\etal}~\cite{STT}
proved that an {\ogf} for these paths is
\begin{equation}\label{eq_dudu}
D(t)=\frac{1+t-t^2-\sqrt{t^4-2t^3-5t^2-2t+1}}{2t}.
\end{equation}
To do so, they found two equations relating them with Dyck paths
that start with a low peak~$\U\D$. Using Lagrange's inversion formula,
they also computed the number of $\dudu$-avoiding Dyck paths of
semilength~$n$ (see also A102407)
$$
d_n=\sum_{j=0}^{\lfloor\frac{n}{2}\rfloor}
\frac{1}{n-j}\binom{n-j}{j}
\sum_{i=0}^{n-2j}\binom{n-2j}{i}\binom{j+i}{n-2j-i+1},
$$
where $n\ge 1$. Letting $d_0=1$ and applying
Proposition~\ref{prop_prim_to_full}, we obtain
\begin{align*}
|\Modasc_n(312)|&=\sum_{k=1}^n\binom{n-1}{k-1}|\Prim_k(312)|\\
&=\sum_{k=1}^n\binom{n-1}{k-1}d_{k-1},
\end{align*}
which gives the sequence
$$
\left(|\Modasc_n(312)|\right)_{n\ge 0}=
1, 1, 2, 5, 14, 43, 142, 495, 1796, 6715, 25692,\dots.
$$
At present, these numbers do not appear in the OEIS~\cite{Sl}.
Combining Equation~\eqref{eq_prim_to_full} with Equation~\eqref{eq_dudu},
we obtain an {\ogf} for $\Modasc_{312}$:
$$
\Modasc_{312}(t)=
\frac{1}{2}\left(3+\frac{t}{1-t}-\frac{t^2}{(1-t)^2}-
\sqrt{\frac{1-6t+7t^2-2t^3+t^4}{(1-t)^4}}\right).
$$
A more direct method to determine $D(t)$ is illustrated below.
The main advantage of our construction is that it relies on a
combinatorial decomposition of $\dudu$-avoiding Dyck paths which we
can replicate on $\Prim(312)$ to define a bijection between these
two structures.
Any nonempty $\dudu$-avoiding Dyck path~$P$ that hits
the $x$-axis $k$ times, $k\ge 1$, decomposes as
\begin{equation}\label{eq_decomp_dudu}
P=\U Q_1\D\;\U Q_2^+\D\;\cdots\;\U Q_{k-1}^+\D\;\U Q_k\D,
\end{equation}
where each factor $Q_i$ is a $\dudu$-avoiding Dyck path; further,
all the factors except for~$Q_1$ and $Q_k$ must be nonempty, as
denoted by the superscript~``$+$''.
Hence $D=D(t)$ satisfies the combinatorial equation
\begin{align*}
D\;&=\;\overbrace{1}^{\text{empty path}}\;+\;
\overbrace{tD}^{k=1}\;+\;
\overbrace{t^2D^2\sum_{k\ge 0}\left[t(D-1)\right]^k}^{k\ge 2}\\
&=\; 1+tD+\frac{t^2D^2}{1-t(D-1)},
\end{align*}
whose solution is given by the {\ogf} of Equation~\ref{eq_dudu}.

To obtain an analogous decomposition on $\Prim(312)$, we shall
decompose primitive, $312$-avoiding modified ascent sequences by
highlighting their copies of~$1$---something we have already done for
the pattern~$122$ in Section~\ref{section_122}. First, we collect
some geometric properties of $\Prim(312)$ in the next proposition.

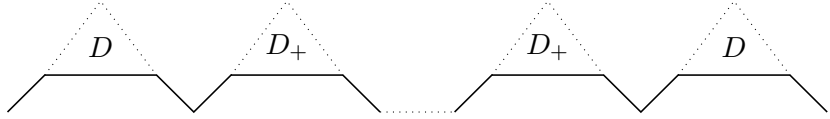
\begin{figure}
\centering
\begin{tikzpicture}[scale=0.49, baseline=20pt]
	\draw[semithick](0,0)--(1,1)--(4,1)--(5,0)--(6,1)--(9,1)--(10,0);
	\draw[semithick](12,0)--(13,1)--(16,1)--(17,0)--(18,1)--(21,1)--(22,0);
	\draw[dotted] (1,1)--(2.5,3)--(4,1);
	\draw[dotted] (6,1)--(7.5,3)--(9,1);
	\draw[dotted] (13,1)--(14.5,3)--(16,1);
	\draw[dotted] (18,1)--(19.5,3)--(21,1);
	\draw[dotted] (10,0)--(12,0);
	\node at (2.5,1.75){$D$};
	\node at (7.5,1.75){$D_+$};
	\node at (14.5,1.75){$D_+$};
	\node at (19.5,1.75){$D$};
\end{tikzpicture}
\caption{Decomposition of a $\dudu$-avoiding Dyck path that hits
the $x$-axis at least twice. Here, $D$ denotes a generic
$\dudu$-avoiding path, while $D_+$ denotes a nonempty one.}\label{figure_dudu}
\end{figure}

\begin{proposition}\label{prop_312}
Let $x\in\Prim_n(312)$, with~$n\ge 1$. Write
$$
x=1B_1\;1B_2\;\cdots\;1B_{k-1}\;1B_k,
$$
where $k\ge 1$ is the number of copies of~$1$ contained in~$x$.
For $i\in[k]$, let $m_i=\max(B_i)$ and denote by $\ell_i$ the
leftmost entry in $B_i$. Then, for each $i\le k-1$:
\begin{enumerate}
\item $B_i\neq\emptyset$.
\item $B_{i+1}\ge m_i$; that is, $a\ge m_i$ for each $a\in B_{i+1}$.
\item $\ell_{i+1}=1+m_i$.
\item Let $\bar{B}_1$ be obtained by subtracting $1$ to each
entry of $B_1$. Then $\bar{B}_1\in\Prim(312)$.
\item Let $\tilde{B}_i$ be obtained by subtracting $m_{i-1}-1$ to each
entry of $B_i$, for $i=2,\dots,k$. Then $1\tilde{B}_i\in\Prim(312)$.
\end{enumerate}
\end{proposition}
\begin{proof}
\begin{enumerate}
\item This claim follows immediately since we are assuming~$x$
to be primitive.
\item An entry $a\in B_{i+1}$, $a<m_i$, would realize an occurrence
$m_i1a$ of~$312$, which is impossible.
\item In a $312$-avoiding modified ascent sequence, all the ascent
tops must be in (strictly) increasing order from left to right.
Indeed, if $x_{j_1}>x_{j_2}$ were ascent tops with $j_1<j_2$, then
$x_{j_1}x_{j_2-1}x_{j_2}$ would be an occurrence of $312$. Now, 
recall that the set $\asctops(x)=\nub(x)$ contains exactly one
copy of each integer from $1$ to $\max(x)$. Further, $m_i$ is the
rightmost (and thus largest) ascent top in $B_i$, while~$\ell_{i+1}$ is
the leftmost (and thus smallest) ascent top in $B_{i+1}$.
The desired claim follows immediately.
\item All the values between~$2$ and $\max(B_1)$ appear in~$B_1$ due
to what proved in Item~3. Note also that the leftmost entry
of~$\bar{B}_1$ is equal to~$x_2-1=2-1=1$. Thus $\bar{B}_1$ is a
Cayley permutation on~$[\max(B_1)-1]$ that starts with~$1$.
Since $\nub(\bar{B}_1)=\asctops(\bar{B}_1)$ and
$\bar{B}_1$ avoids~$312$, the word $\bar{B}_1$ is a primitive,
$312$-avoiding modified ascent sequence.
\item The proof of this item is analogous to the previous one.
The only difference is that, the correct quantity to subtract in order
to rescale the entries of $B_i$ properly is $m_{i-1}-1=\ell_{i}-2$.
Indeed, let $a\in B_i$. Then $a\ge m_{i-1}$ due to Item~2, and
$$
a-(m_{i-1}-1)\ge m_{i-1}-m_{i-1}+1=1.
$$
Similarly, $\ell_i=m_{i-1}+1$ due to Item~3, and
$$
\ell_i-(m_{i-1}-1)=m_{i-1}+1-m_{i-1}+1=2.
$$
As a result, all the values between~$1$ and~$m_i-m_{i-1}+1$ appear
in~$1\tilde{B}_i$; that is, the word~$1\tilde{B}_i$ is a Cayley
permutation on~$[m_i-m_{i-1}+1]$. More specifically, in analogy with
what observed for $B_1$, it is a primitive, $312$-avoiding modified
ascent sequence.
\end{enumerate}
\end{proof}

Keeping the same notations of Proposition~\ref{prop_312}, every
nonempty $x\in\Prim(312)$ that contains~$k\ge 1$ copies of~$1$
decomposes as
$$
x=1B_1\;1B_2\;\cdots\;1B_{k-1}\;1B_k,
$$
where $B_i$ is nonempty for $i\le k-1$, the leftmost
block~$B_1$ sastisfies $\bar{B}_1\in\Prim(312)$, and
$1\tilde{B}_i\in\Prim(312)$ for each $i\ge 2$.
As an example, let $x=123432561761897$. Note that $x\in\Prim(312)$.
Then~$x$ decomposes as
$$
x\;=\;
1\underbrace{2343256}_{B_1}\;
1\underbrace{76}_{B_2}\;
1\underbrace{897}_{B_3},
$$
where
\begin{center}
\begin{tabular}{l}
$\;\;\bar{B}_1=1232145$,\\
$1\tilde{B}_2=121$,\\
$1\tilde{B}_3=1231$.
\end{tabular}
\end{center}
On the other hand, given any such sequence
$\bar{B}_1,1\tilde{B}_2,\dots,1\tilde{B}_k$ of primitive,
$312$-avoiding modified ascent sequences, one uniquely
reconstruct~$x=1B_1\;1B_2\;\cdots\;1B_k$ by suitably rescaling
the entries of the blocks~$\bar{B}_1,\tilde{B}_2,\dots,\tilde{B}_k$
as in the last two items of Proposition~\ref{prop_312}; that is,
by adding~$1$ to each entry of $\bar{B}_1$, and $m_{i-1}-1$
to each entry of $\tilde{B}_i$, where $m_{i-1}$ is the maximum
of $B_{i-1}$ and $i\ge 2$.

We now have all the ingredients to define a bijection between
$\Prim_{n+1}(312)$ and the set of $\dudu$-avoiding Dyck paths
of semilength~$n$. Given $x\in\Prim(312)$, we define the path
$\phi(x)=P$ recursively by letting $\phi(\emptyset)=\phi(1)=\emptyset$
and, if $x=1B_11B_2\cdots1B_{k-1}1B_k$ has length two
or more,
$$
\begin{array}{cccccc}
\phi(1B_1 &1B_2 &\dots & 1B_{k-1} & 1B_k)\\
=\U\phi(\bar{B}_1)\D
&\U\phi(1\tilde{B}_2)\D&\dots
&\U\phi(1\tilde{B}_{k-1})\D
&\U\phi(1\tilde{B}_k)\D.
\end{array}
$$
For instance, we have
\begin{align*}
&\phi(12)=\U\phi(\bar{2})\D=\U\phi(1)\D =\U\D;\\
&\phi(121)=\U\phi(\bar{2})\D\U\phi(\emptyset)\D=\U\D\U\D;\\
&\phi(123)=\U\phi(\bar{23})\D=\U\phi(12)\D=\U\U\D\D.
\end{align*}
The Dyck path (of semilength~$14$) associated with the
sequence~$x=123432561761897$ (of semilength~$15$) of the
previous example is depicted in Figure~\ref{figure_dudupath}.
The leftmost factor of the path is obtained from $1B_1=12343256$,
recursively, as
\begin{align*}
\U\phi(\bar{B}_1)\D &= \U\left(\phi(1232145)\right)\D\\
&= \U\left(
\left(\U\phi(121)\D\right)
\left(\U\phi(123)\D\right)
\right)\D\\
&= \U\left(
\left(\U\U\D\U\D\D\right)
\left(\U\U\U\D\D\D\right)
\right)\D\\
&=\U^3\D\U\D^2\U^3\D^4.
\end{align*}
By Proposition~\ref{prop_312}, the path~$P$ satisfies the
decomposition determined by Equation~\ref{eq_decomp_dudu}.
In particular, for $2\le i<k$ the path~$\phi(1\tilde{B}_i)$ is
not empty since~$B_i$ is nonempty. Hence~$P$ is a $\dudu$-avoiding
Dyck path. Note also that the semilength of~$P=\phi(x)$ is equal to
one less than the length of~$x$; the shift in length is a result
of having mapped the leftmost
block~$1B_1$ to $\U\phi(\bar{B}_1)\D$; on the other hand, the
semilength of every other factor $\U\phi(1\tilde{B}_i)\D$ matches
the length of the corresponding block~$1B_i$ of~$x$.
Due to the above discussion, the map $\phi$ defined this way
is a bijection between $\Prim_{n+1}(312)$ and the set of
$\dudu$-avoiding Dyck paths of semilength~$n$.

\begin{figure}
\centering
\begin{tikzpicture}[scale=0.49, baseline=20pt]
	\draw[semithick] (0,0)--(3,3)--(4,2)--(5,3)--(7,1)--(10,4)--(14,0)
	--(16,2)--(17,1)--(18,2)--(20,0)--(23,3)--(25,1)--(26,2)--(28,0);
	\draw[semithick] (0,0)--(28,0);
	\draw[dotted] (1,1)--(13,1);
	\draw[dotted] (15,1)--(19,1);
	\draw[dotted] (21,1)--(27,1);
	\node[] at (7,-0.5){$12343256$};
	\node[] at (17,-0.5){$176$};
	\node[] at (24,-0.5){$1897$};
	\filldraw (0,0) circle (3pt);
	\filldraw (1,1) circle (3pt);
	\filldraw (2,2) circle (3pt);
	\filldraw (3,3) circle (3pt);
	\filldraw (4,2) circle (3pt);
	\filldraw (5,3) circle (3pt);
	\filldraw (6,2) circle (3pt);
	\filldraw (7,1) circle (3pt);
	\filldraw (8,2) circle (3pt);
	\filldraw (9,3) circle (3pt);
	\filldraw (10,4) circle (3pt);
	\filldraw (11,3) circle (3pt);
	\filldraw (12,2) circle (3pt);
	\filldraw (13,1) circle (3pt);
	\filldraw (14,0) circle (3pt);
	\filldraw (15,1) circle (3pt);
	\filldraw (16,2) circle (3pt);
	\filldraw (17,1) circle (3pt);
	\filldraw (18,2) circle (3pt);
	\filldraw (19,1) circle (3pt);
	\filldraw (20,0) circle (3pt);
	\filldraw (21,1) circle (3pt);
	\filldraw (22,2) circle (3pt);
	\filldraw (23,3) circle (3pt);
	\filldraw (24,2) circle (3pt);
	\filldraw (25,1) circle (3pt);
	\filldraw (26,2) circle (3pt);
	\filldraw (27,1) circle (3pt);
	\filldraw (28,0) circle (3pt);
\end{tikzpicture}
\caption{Dyck path associated with the sequence $x=123432561761897$.
}\label{figure_dudupath}
\end{figure}
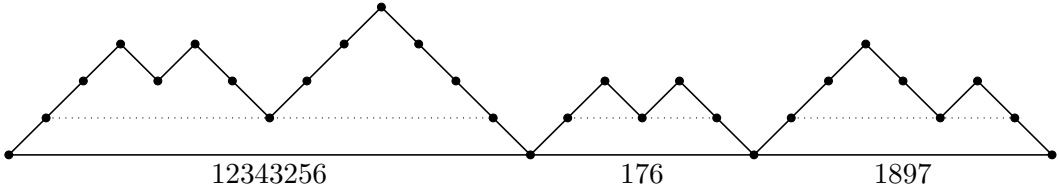

\begin{remark*}
We end this section with a numerical remark. Bao {\etal}~\cite{BCCGZ}
have recently showed a bijection between $\dudu$-avoiding Dyck paths
and the set of permutations that are sorted by the $(132,321)$-machine.
They also characterized these permutations as
$$
\mathrm{Sort}(132,321)=
\Sym\left(132,\;
\begin{tikzpicture}[scale=0.4, baseline=20.5pt]
\fill[NE-lines] (2,0) rectangle (3,4);
\fill[NE-lines] (0,2) rectangle (4,3);
\draw [semithick] (0.001,0.001) grid (3.999,3.999);
\filldraw (1,1) circle (6pt);
\filldraw (2,3) circle (6pt);
\filldraw (3,2) circle (6pt);
\end{tikzpicture}
\right).
$$
Using the BiSC algorithm~\cite{Per} and Theorem~\ref{thm_transport},
we can conjecture that
$$
\std\left(\Prim(312)\right)=
\Sym\left(2413,\;
\begin{tikzpicture}[scale=0.4, baseline=20.5pt]
\fill[NE-lines] (1,0) rectangle (2,4);
\fill[NE-lines] (0,1) rectangle (4,2);
\draw [semithick] (0.001,0.001) grid (3.999,3.999);
\filldraw (1,3) circle (6pt);
\filldraw (2,2) circle (6pt);
\filldraw (3,1) circle (6pt);
\end{tikzpicture},\;
\begin{tikzpicture}[scale=0.4, baseline=20.5pt]
\fill[NE-lines] (2,3) rectangle (3,4);
\draw [semithick] (0.001,0.001) grid (3.999,3.999);
\filldraw (1,3) circle (6pt);
\filldraw (2,1) circle (6pt);
\filldraw (3,2) circle (6pt);
\end{tikzpicture}
\right).
$$
Can the Wilf-equivalence between $\mathrm{Sort}_n(132,321)$ and
$\std\left(\Prim_{n+1}(312)\right)$ be explained more transparently?
\end{remark*}

\subsection{Patterns~$1213$ and $1312$}\label{section_1213_1312}

Modified ascent sequences are subject to the geometric constraints
established by Proposition~\ref{prop_modasc_prop}. This explains the
presence of patterns $x,y\in\Cay$, $x\neq y$, equivalent in the sense
that $\Modasc(x)=\Modasc(y)$. For instance, we~\cite{Ce2} have proved
that
$$
\Modasc(212)=\Modasc(1212)=\Modasc(2132)=\Modasc(12132).
$$
The following result has the same flavor.

\begin{proposition}
We have
$$
\Modasc(213)=\Modasc(1213)
\quad\text{and}\quad
\Modasc(312)=\Modasc(1312).
$$
\end{proposition}
\begin{proof}
Clearly, $\Modasc(213)\subseteq\Modasc(1213)$. Conversely,
let $x\in\Modasc$ and suppose that~$x$ contains~$213$. Let $x_ix_jx_k$
be an occurrence of $213$ in $x$. Without losing generality, we can
assume that~$x_j$ is the smallest entry between $x_i$ and $x_k$, taking the
leftmost one in case of ties. Due to our choice, we have $x_{j-1}>x_j$.
Hence $x_j\notin\asctops(x)=\nub(x)$. Further, if $x_{\ell}$ is the
leftmost copy of $x_j$ in $x$, then it must be $\ell<i$. Finally,
we obtain the desired occurrence~$x_{\ell}x_ix_jx_k$ of~$1213$.
To prove the remaining equality, simply replace~$213$ with~$312$
and use the same argument.
\end{proof}

\section{Patterns $221$ and $2321$}\label{section_221_2321}

Recall from Section~\ref{section_primitive} that any $x\in\Modasc$
is obtained (uniquely) from a primitive modified ascent
sequence~$w$ by suitably inserting some flat steps. If~$y$ is primitive
and $w$ avoids~$y$, then inserting flat steps does not create
occurrences of~$y$ in $x$. In other words, all the positions between
two consecutive entries of $w$ are active sites in this sense.
It is clear that the same mechanic fails if $y$ is not primitive.
For instance, the primitive sequence~$w=12$ avoids~$122$, but the
insertion of a flat step at the end gives~$x=122$ (which contains~$122$).
In this section, however, we are able to slightly tweak this approach
by computing the distribution of active sites on~$\Prim(221)$.
En passant, we enumerate $\Modasc(2321)$, finally settling the
remaining case of a conjecture by Duncan and Steingr\'imsson~\cite{DS}.

\begin{proposition}\label{prop_221_2321}
For each $n\ge 0$, we have $\Prim_n(221)=\Prim_n(2321)$.
Furthermore,
$$
\std\left(\Prim_n(221)\right)=1\oplus\Sym_{n-1}(\dashpatt),
\quad\text{where}\quad
\dashpatt=
\begin{tikzpicture}[scale=0.4, baseline=20.5pt]
\fill[NE-lines] (1,0) rectangle (2,4);
\draw [semithick] (0.001,0.001) grid (3.999,3.999);
\filldraw (1,3) circle (6pt);
\filldraw (2,2) circle (6pt);
\filldraw (3,1) circle (6pt);
\end{tikzpicture}.
$$
\end{proposition}
\begin{proof}
Let us start with the equality $\Prim_n(221)=\Prim_n(2321)$.
The inclusion $\Prim_n(221)\subseteq\Prim_n(2321)$ is trivial. To prove
the other inclusion, suppose that~$x$ contains~$221$ and let $x_ix_jx_k$
be an occurrence of $221$ in~$x$. Note that $x_j\notin\nub(x)=\asctops(x)$.
Since~$x$ is primitive, it must be $x_{j-1}>x_j$. Thus
$x_ix_{j-1}x_jx_k\simeq 2321$, as wanted.\\
Next we prove that
$\std\bigl(\Prim_n(221)\bigr)=1\oplus\Sym_{n-1}(\dashpatt)$.
Let $x\in\Prim_n$ and let $p=\std(x)$. We show that $x\ge 221$ if and
only if $p\ge\dashpatt$. Initially, suppose that
$x\ge 221$. As showed above, $x$ contains an occurrence
$x_ix_{j-1}x_jx_k$ of $2321$. Then, by Lemma~\ref{lemma_std_ascdes},
we have~$p_{j-1}p_jp_k\simeq\dashpatt$.
Conversely, suppose that~$p$ contains an occurrence $p_{j-1}p_jp_k$
of~$\dashpatt$. By the same lemma, it must be $x_{j-1}>x_j>x_k$,
and thus~$x_j\notin\asctops(x)=\nub(x)$. By taking the leftmost
copy of $x_j$ in $x$, say $x_i$, we get the desired occurrence
$x_ix_jx_k$ of $221$.
\end{proof}

Duncan and Steingr\'imsson~\cite{DS} conjectured that modified ascent
sequences avoiding any of the patterns~$212$, $1212$, $2132$, $2213$,
$2231$ and $2321$ are counted by the Bell numbers. More specifically,
they suggested that the distribution of the number of ascents was given
by the reverse of the distribution of blocks on set partitions.
The current author~\cite{Ce2} settled this conjecture for all the
patterns except for~$2321$, which we are finally able to solve here.

\begin{proposition}\label{conj_DS}
The cardinality of $\Modasc_n(2321)$ is equal to the $n$th Bell number.
\end{proposition}
\begin{proof}
Claesson~\cite[Prop.~2]{Cl} showed that $|\Sym_n(\dashpatt)|=b_n$,
where~$b_n$ is the $n$th Bell number. Here, we have:
\begin{align*}
|\Modasc_n(2321)|&=\sum_{k=1}^n\binom{n-1}{k-1}|\Prim_k(2321)|
& \text{by Proposition~\ref{prop_prim_to_full}}\\
&=\sum_{k=1}^n\binom{n-1}{k-1}|1\oplus\Sym_{k-1}(\dashpatt)|
& \text{by Proposition~\ref{prop_221_2321}}\\
&=\sum_{k=1}^n\binom{n-1}{k-1}b_{k-1}& \text{by Claesson~\cite{Cl}}\\
&=b_n,
\end{align*}
where the last equality is a well known recurrence for the Bell numbers.
\end{proof}

Next, we recall a useful bijection\footnote{Claesson's map is
defined on permutations avoiding the reverse of $\dashpatt$.} between
set partitions of $[n]$ and $\Sym_n(\dashpatt)$ originally discovered
by Claesson~\cite[Prop.~2]{Cl}.
Given a partition $\beta$ of $[n]$, the \emph{standard representation}
of $\beta$ is obtained by writing
\begin{itemize}
\item[(i)] each block with its least element last, and the other
elements in increasing order;
\item[(ii)] blocks in increasing order of their least element, with
dashes separating two consecutive blocks.
\end{itemize}
For instance, the standard representation of
$$
\beta=\{\{1,3,6\},\{2,7\},\{4\},\{5,8,9\}\}
\quad\text{is}\quad
\beta=361\mhyphen 72\mhyphen 4\mhyphen 895.
$$
Then $\beta$ is associated with the $(\dashpatt)$-avoiding
permutation $p$ obtained by writing $\beta$ in standard representation,
and erasing the dashes. The set partition~$\beta$ in the previous example
is associated with $p=361724895$.
Claesson~\cite[Prop.~3]{Cl} showed that the number of
$(\dashpatt)$-avoding permutations of length~$n$ with~$j$ right-to-left
minima is equal to the $(n,j)$th Stirling number of the second
kind~$S(n,j)$. The next lemma follows in a similar fashion.

\begin{lemma}\label{lemma_des_notsingl}
Let $p\in\Sym(\dashpatt)$ be associated with the set partition~$\beta$
via Claesson's bijection. Then $\des(p)$ is equal to the number
of blocks of $\beta$ that are not singletons.
\end{lemma}
\begin{proof}
There is a descent~$p_{i-1}>p_{i}$ in~$p$ if and only if~$p_{i}$
is the minimum of a block of~$\beta$ that has size two or more.
\end{proof}

From now on, let $\Par[n]$ denote the set of set partitions over $[n]$
and let
$$
p_{n,i}=|\{\beta\in\Par[n]:
\text{$\beta$ has $i$ blocks that are not singletons}\}|.
$$
The coefficients $p_{n,i}$ (see also A124324) are related to the
Stirling numbers of the second kind by the following proposition.

\begin{proposition}\label{prop_stirl2}
We have
$$
S(n,n-h)=\sum_{i=h+1}^n\binom{n-1}{n-i}p_{i-1,i-1-h}.
$$
\end{proposition}
\begin{proof}
Let $\beta\in\Par[n]$ be a set partition with $n-h$ blocks.
Then $\beta$ consists of
\begin{itemize}
\item a block $A$ that contains $1$;
\item some singletons $\{s_1\},\dots,\{s_{n-i}\}$;
\item some blocks $B_1,\dots,B_{i-1-h}$ of size at least two,
\end{itemize}
where $n-i\le n-h-1\iff i\ge h+1$. Alternatively, $\beta$ is
uniquely determined by choosing
\begin{itemize}
\item the singletons $\{s_1\},\dots,\{s_{n-i}\}$, which can be done
in $\binom{n-1}{n-i}$ ways;
\item a set partition $\alpha$ of the remaining $n-1-(n-i)=i-1$ elements,
excluding~$1$, with $i-1-h$ blocks that are not singletons; here,
the singletons of $\alpha$ shall form the block~$A$, together with~$1$,
while the $i-1-h$ non-singletons block are $B_1,\dots,B_{i-1-h}$.
\end{itemize}
More schematically,
\begin{align*}
\beta&=
\{\overbrace{\{1,a_1,\dots,a_{\ell}\}}^A,\;
\{s_1\},\dots,\{s_{n-i}\},\;
B_1,\dots,B_{i-1-h}\};
\\
\alpha&=\{\{a_1\},\dots,\{a_{\ell}\},\;B_1,\dots,B_{i-1-h}\}.
\end{align*}
Since there are exactly $p_{i-1,i-1-h}$ partitions $\alpha$ as above,
our claim follows.
\end{proof}

\begin{remark}\label{remark_egf}
A weighted exponential generating function for the coefficients
$p_{n,i}$ is
\begin{align*}
P_s(t)=\sum_{n\ge 0}\left(\sum_{i\ge 0}p_{n,i}s^i\right)\frac{t^n}{n!}
&=\exp(s(e^t-t-1)+t),
\end{align*}
obtained my marking every non-singleton block with~$s$.
Proposition~\ref{prop_stirl2} could be established algebraically
by observing that
\begin{align*}
\sum_{n\ge 0}\left(\sum_{i\ge 0}S(n,i)s^i\right)\frac{t^n}{n!}
&=\exp(s(e^t-1))\\
&=P_s(t)\cdot\exp\bigl(t(s-1)\bigr).
\end{align*}
The proof is rather technical, and it can be found in the Appendix.
\end{remark}

Now, our goal is to prove that the number of $2321$-avoiding modified
ascent sequences with $h$ ascents is equal to $S(n,n-h)$.
By Proposition~\ref{prop_221_2321} and Lemma~\ref{lemma_std_ascdes},
\begin{align*}
|\{x\in\Prim_{n}(2321):\asc(x)=h\}|&=
|\{p\in 1\oplus\Sym_{n-1}(\dashpatt):\asc(p)=h\}|\\
&=|\{p\in\Sym_{n-1}(\dashpatt):\asc(p)=h-1\}|\\
&=|\{p\in\Sym_{n-1}(\dashpatt):\des(p)=n-h-1\}|\\
&=p_{n-1,n-h-1},
\end{align*}
where the last step follows by Lemma~\ref{lemma_des_notsingl}.
Finally, since the insertion of any number of flat steps preserves
the number of ascents, by Proposition~\ref{prop_prim_to_full}
we have
\begin{align*}
|\{x\in\Modasc_{n}(2321):\asc(x)=h\}|
&=\sum_{i=h+1}^n\binom{n-1}{i-1}|\{x\in\Prim_{i}(2321):\asc(x)=h\}|\\
&=\sum_{i=h+1}^n\binom{n-1}{i-1}p_{i-1,i-h-1}\\
&=S(n,n-h),
\end{align*}
where the last equality is Proposition~\ref{prop_stirl2}.

Let us now go back to the pattern $221$.

\begin{proposition}
We have
$$
|\Modasc_n(221)|=\sum_{k=1}^n\sum_{i=1}^k S(k-1,i-1)\binom{n-1-k+i}{i-1}.
$$
\end{proposition}
\begin{proof}
Let $w\in\Prim_k(221)$. For $i=1,2,\dots,k$, we say that $i$ is
an \emph{active site} if inserting a flat step $a=w_i$ in the position
between $w_i$ and $w_{i+1}$ (or after $w_k$, if $i=k$) does not create
an occurrence of~$221$; that is, if
$$
w_1\cdots w_i\;w_i\;w_{i+1}\cdots w_k
\quad\text{avoids}\;221.
$$
It is easy to see that $i$ is active if and only if
$w_i$ is a weak right-to-left minimum.
Specifically, if $w\in\Prim_k(221)$ has $i$ weak right-to-left minima,
then $w$ has $k-i$ sites that are not active. Now, any sequence
$x\in\Modasc_n(221)$ is obtained from some $w\in\Prim_k(221)$,
with~$1\le k\le n$, by inserting $n-k$ flat steps among a total of $n-1$
positions (recall that $x_1=1$ is forced), minus the
$k-|\wrtlmin(w)|$ sites that are not active.
Thus, we can adapt the formula of Proposition~\ref{prop_prim_to_full}
accordingly to obtain
\begin{align*}
|\Modasc_n(221)|
&=\sum_{k=1}^n\sum_{i=1}^k
|\{w\in\Prim_k(221):\#\wrtlmin(w)=i\}|\binom{n-1-k+i}{n-k}\\
&=\sum_{k=1}^n\sum_{i=1}^k
|\{w\in\Prim_k(221):\#\wrtlmin(w)=i\}|\binom{n-1-k+i}{i-1}.
\end{align*}
Finally, by Proposition~\ref{prop_221_2321} and Lemma~\ref{lemma_std_props},
\begin{align*}
|\{w\in\Prim_k(221):\#\wrtlmin(w)=i \}|
&=|\{p\in\Sym_{k-1}(\dashpatt):\#\rtlmin(p)=i-1\}|\\
&=S(k-1,i-1),
\end{align*}
where the last equality is once again due to Claesson~\cite[Prop.~3]{Cl}.
\end{proof}

For $n\ge 0$, the sequence $|\Modasc_n(221)|$ starts with
$1,1,2,5,14,44,155,607,2617$ and does not appear in the OEIS~\cite{Sl}.

\section{Final remarks and future directions}\label{section_final}

In this paper, we enumerated the sets $\Modasc(y)$ for every
pattern~$y$ of length at most three, except for
$y\in\{111,211\}$. Interestingly, both patterns are currently open on
plain ascent sequences too. We have reported the corresponding data
in Table~\ref{table_unsolved_patts}, together with longer patterns
we were not able to solve despite the promising evidence.
We end with a list of suggestions for future work.

\begin{table}
\centering
\def\arraystretch{1}
\begin{tabular}{lll}
\toprule
$y$ & $|\Modasc_n(y)|$ & $|\Prim_n(y)|$\\
\midrule
111 & \small{1, 2, 4, 10, 29, 97, 367, 1550} &
\small{1, 1, 2, 5, 14, 46, 172, 718, 3317, 16796}\\
211,1223 & A047970? & \small{1, 1, 2, 5, 14, 44, 153, 581, 2385}\\
1324,1342 & A007317? & Catalan?\\
4321 & \small{1, 2, 5, 15, 53, 217, 1008, 5188} &
\small{1, 1, 2, 5, 16, 61, 265, 1267}\\
\bottomrule
\end{tabular}
\caption{Unsolved patterns.}\label{table_unsolved_patts}
\end{table}

\begin{itemize}
\item In Section~\ref{section_1213_1312}, we proved that
$\Modasc(213)=\Modasc(1213)$ and $\Modasc(312)=\Modasc(1312)$.
Are there any other examples of patterns that are equivalent in this
sense? More in general, can we characterize all the sets of
equivalent patterns?
\item There is only one Cayley permutation~$x$ whose standardization is
the decreasing permutation $p=k\cdots 21$, namely $x=p$.
Thus, by Theorem~\ref{thm_transport},
$$
\std\bigl(\Prim(k\cdots 21)\bigr)=\Omega(k\cdots 21).
$$
We solved the case $k=3$ in Section~\ref{section_213_231_321}.
Can we tackle the general case with the same approach?
In a similar fashion, can we generalize what we proved in
Section~\ref{section_123} for $\Prim(123)$ to $\Prim(12\cdots k)$?
\item Among the unsolved patterns, it appears that
$$
\Modasc_n(211)=\Modasc_n(1223)
\quad\text{and}\quad
\Prim_n(211)=\Prim_n(1223)
$$
at least up to $n=10$. Can we prove that the equalities hold
for every~$n$?
\item Note the following two, rather curious, chains of
inclusions:
$$
\begin{array}{ccccc}
\Prim(213) & \subseteq & \Modasc(213) & \subseteq & \Modasc(1324);\\
\Prim(231) & \subseteq & \Modasc(231) & \subseteq & \Modasc(1342),\\
\text{(Motzkin)} && \text{(Catalan)} && \text{(A007317?)}
\end{array}
$$
where each term is (counted by) the binomial transform of the term
to its left. Can we use this to count $\Modasc(1324)$ and $\Modasc(1342)$?
Is this phenomenon more general?
\item In Section~\ref{section_122}, we found an {\ogf} for $\Modasc(122)$.
It appears that
$$
\Modasc_{122}(t)=(1-t)\Modasc_{211}(t),
$$
where $211$ is one of the patterns we could not solve. Why?
\item We have decided to leave the study of modified ascent sequences
avoiding pairs (or sets) of patterns for a future investigation.
An example that is particularly dear to us is the
following. The Burge transpose~\cite{CC} maps bijectively
$\Modasc(2312,3412)$ to the set~$\F(3412)$ of Fishburn
permutations avoiding~$3412$. A numerical analysis
suggests that the pair of statistics right-to-left maxima and
right-to-left minima on $\F(3412)$ is equidistributed with
the pair left-to-right maxima and right-to-left maxima
over the set of $312$-sortable permutations~\cite{CCF}.
The first terms of the arising counting sequence match
A202062~\cite{Sl}. Currently, no formula or generating function
for~A202062 is known. An asymptotic analysis of this sequence has
been conducted recently by Conway \etal~\cite{CCEG}.
\end{itemize}

\textbf{Acknowledgements.} The author is grateful to Anders Claesson
for suggesting that a proof of Proposition~\ref{prop_stirl2}
could be obtained via exponential generating functions, as
well as for pointing out the Schensted reference.

\appendix

\section*{Appendix}

Let
$$
P_s(t)=\sum_{n\ge 0}\left(\sum_{i\ge 0}p_{n,i}s^i\right)\frac{t^n}{n!}
\quad\text{and}\quad
Q_s(t)=\sum_{n\ge 0}\left(\sum_{i\ge 0}S(n,i)s^i\right)\frac{t^n}{n!}
$$
be the (weighted) exponential generating functions of the coefficients
$p_{n,i}$, defined in Section~\ref{section_221_2321}, and the
Stirling numbers of the second kind $S(n,i)$, respectively.
We give an algebraic proof of the formula
$$
S(n,n-h)=\sum_{i=h+1}^n\binom{n-1}{n-i}p_{i-1,i-1-h},
$$
which we proved combinatorially in Proposition~\ref{prop_stirl2}.
Recall from Remark~\ref{remark_egf} that
\begin{align*}
Q_s(t) &= P_s(t)\cdot\exp\bigl(t(s-1)\bigr)\\
\iff Q_s(t)\cdot\exp(t) &= P_s(t)\cdot\exp(st).
\end{align*}
Let us expand both sides of the latter equation. First,
\begin{align*}
Q_s(t)\cdot\exp(t) &=
\left(\sum_{n\ge 0}\left(\sum_{k\ge 0}S(n,k)s^k\right)\frac{t^n}{n!}\right)
\cdot
\left(\sum_{n\ge 0}\frac{t^n}{n!}\right)\\
&=\sum_{n\ge 0}\left(
\sum_{j\ge 0}\binom{n}{j}\sum_{k\ge 0}S(j,k)s^k
\right)\frac{t^n}{n!}\\
&=\sum_{n\ge 0}\left(
\sum_{j\ge 0}\sum_{k\ge 0}\binom{n}{j}S(j,k)s^k
\right)\frac{t^n}{n!}\\
&=\sum_{n\ge 0}\left(
\sum_{k\ge 0}S(n+1,k+1)s^k
\right)\frac{t^n}{n!},
\end{align*}
where at the last step we used that
$$
S(n+1,k+1)=\sum_{j\ge 0}\binom{n}{j}S(j,k).
$$
Secondly,
\begin{align*}
P_s(t)\cdot\exp(st) &=
\left(\sum_{n\ge 0}\left(\sum_{k\ge 0}p_{n,k}s^k\right)\frac{t^n}{n!}\right)
\cdot
\left(\sum_{n\ge 0}s^n\frac{t^n}{n!}\right)\\
&=\sum_{n\ge 0}\left(
\sum_{j\ge 0}\binom{n}{j}\sum_{k\ge 0}p_{j,k}s^ks^{n-j}
\right)\frac{t^n}{n!}\\
&=\sum_{n\ge 0}\left(
\sum_{j\ge 0}\sum_{k\ge 0}\binom{n}{j}p_{j,k}s^{n+k-j}
\right)\frac{t^n}{n!}.
\end{align*}
By comparing the coefficients in front of $x^n/n!$ (and using $\ell$
instead of $k$ in the left-hand sum), we obtain
$$
\sum_{\ell\ge 0}S(n+1,\ell+1)s^\ell
=\sum_{j\ge 0}\left(\sum_{k\ge 0}\binom{n}{j}p_{j,k}\right)s^{n+k-j}.
$$
Finally, since $\ell=n+k-j\iff k=\ell+j-n$, we have
\begin{align*}
S(n+1,\ell+1) &=\sum_{j\ge 0}\binom{n}{j}p_{j,\ell+j-n}\\
&=\sum_{j=n-\ell}^n\binom{n}{j}p_{j,\ell+j-n}\\
\iff S(n,\ell)&= \sum_{j=n-\ell}^{n-1}\binom{n-1}{j}p_{j,(\ell-1)+j-(n-1)}\\
\iff S(n,\ell)&= \sum_{i=n-\ell+1}^{n}\binom{n-1}{i-1}p_{i-1,\ell+(i-1)-n}\\
\iff S(n,n-h) &= \sum_{i=h+1}^{n}\binom{n-1}{i-1}p_{i-1,i-1-h}\\
&=\sum_{i=h+1}^{n}\binom{n-1}{n-i}p_{i-1,i-1-h}.
\end{align*}

\end{document}